\documentclass{amsart}
\usepackage[mathscr]{eucal}
\usepackage{enumerate}

\newtheorem{theorem}{Theorem}[section]
\newtheorem{lemma}[theorem]{Lemma}
\usepackage{xcolor}
\theoremstyle{definition}

\newtheorem{example}[theorem]{Example}

\newtheorem{cor}[theorem]{Corollary}
\newtheorem{prop}[theorem]{Proposition}
\theoremstyle{remark}

\numberwithin{equation}{section}

\begin{document}
	
	
	\title[ Linear maps in $\mathcal{L}(\ell_{\MakeLowercase{p}},\mathcal{Y})$ preserving parallel and TEA pairs] {{ Linear maps in $\mathcal{L}(\ell_{\MakeLowercase{p}},\mathcal{Y}) $ preserving parallel and TEA pairs}}

\author[Arpita Mal]{Arpita Mal}

\address[]{Dhirubhai Ambani University\\ Gandhinagar-382007\\ India.}
\email{arpitamalju@gmail.com}



\subjclass[2020]{Primary 15A86, 47B01, 
	Secondary  46B20}
\keywords{Parallelism; parallel pair preserver; triangle equality attaining pair preserver; linear preserver problem}



\date{}
\maketitle
\begin{abstract}

A pair $(x,y)$ of vectors in a Banach space $\mathcal{Y}$ is said to be a  triangle equality attaining (or TEA)  pair if $\|x+ y\|=\|x\|+\|y\|,$ and a parallel  pair if $\|x+\lambda y\|=\|x\|+\|y\|$ holds for some unimodular scalar $\lambda.$  In this article,  we explore bounded linear maps $T:\ell_p\to \mathcal{Y}$ preserving parallel and TEA pairs. For $p\in(1,\infty)$ all linear maps trivially preserve parallel pairs.  We prove that for $p=\infty,$ if $\ker(T)\neq \{0\},$ then $T$ preserves parallel  pairs if and only if $rank(T)\leq 1$. 
In particular, $T:\ell_\infty\to \ell_1$ preserves parallel (resp. TEA) pairs if and only if $rank(T)\leq 1$ (resp. $T=0$). Analogous characterizations hold if $T$ is defined from $\ell_\infty^n,$ except when $n=2$ and the field is real.  In this specific setting, we further characterize such maps  $T:\ell_\infty^2\to \ell_1.$ \\
Focusing on $p=1$, we establish a necessary condition for the preservation of parallel pairs. Specifically, we characterize invertible parallel pair preservers $T:\ell_1^n\to\ell_\infty^n$, as well as the general class of such maps $T:\ell_1^2 \to \ell_\infty^m,$ revealing the intricate structure inherent to these mappings. Furthermore, we prove that $(0\neq)~T:\ell_1\to \mathcal{Y}$ preserves TEA pairs if and only if $\Lambda=\{i\in \mathbb{N}:Te_i\neq 0\}$ is singleton, where $\mathcal{Y}$ is either strictly convex or $\ell_\infty^m$ over the complex field. Finally we characterize the TEA pair preservers $T:\ell_1^2\to \ell_\infty^m$ over the real field.

\end{abstract}

\section{Introduction}
Linear preserver problems represent one of the most active areas of research across diverse branches of mathematics, including matrix theory, Banach space theory, operator theory, operator algebra, etc. The purpose of this article is to explore linear maps preserving certain norm equality conditions. 
Suppose $\mathcal{X},\mathcal{Y}$ are  Banach spaces over the field $\mathbb{F}.$ Unless otherwise specified, assume $\mathbb{F}\in\{\mathbb{R}, \mathbb{C}\}.$ Given $x,y\in \mathcal{X},$ we say that the pair $(x,y)$ is a triangle equality attaining pair in $\mathcal{X}$, in short TEA pair, if  
$$\|x+y\|=\|x\|+\|y\|,$$
and  $(x,y)$ is a parallel pair in $\mathcal{X}$, if there exists $\lambda\in \mathbb{T}:=\{\mu\in \mathbb{F}:|\mu|=1\}$ such that 
 $$\|x+\lambda y\|=\|x\|+\|y\|.$$
 If $(x,y)$ is a parallel pair in $\mathcal{X},$ then so are the pairs $(y,x),$ and $(\alpha x,\beta y)$ for all $\alpha,\beta\in \mathbb{F}.$ If $\mathcal{X}$ is strictly convex, then $(x,y)$ is a parallel pair if and only if $\{x,y\}$ is linearly dependent. \\
 
Suppose $\mathcal{L}(\mathcal{X},\mathcal{Y})$ is the space of all bounded linear maps from $\mathcal{X}$ to $\mathcal{Y}.$   A linear map $T:\mathcal{X}\to \mathcal{Y}$ preserves parallel pairs, if for all $ x,y\in \mathcal{X},$
\[(x,y) \text{ is a parallel pair in } \mathcal{X} \quad \Rightarrow \quad  (Tx,Ty) \text{ is a parallel pair in } \mathcal{Y},\] 
whereas $T$ preserves TEA pairs, if for all $ x,y\in \mathcal{X},$
\[(x,y) \text{ is a TEA pair in } \mathcal{X} \quad \Rightarrow \quad  (Tx,Ty) \text{ is a TEA pair in } \mathcal{Y}.\] 
Although TEA pair preserving maps are also parallel pair preserving, the converse need not be true.
Since any two linearly dependent vectors always form a parallel pair, rank one linear maps always preserve parallel pairs. Note that, if there is an isometry $T:\mathcal{X}\to \mathcal{Y},$ then clearly $T$ preserves parallel pairs. Due to the pivotal role played by such maps in understanding the geometry of the underlying Banach spaces and the corresponding operator spaces,  these maps are recently emerging as a focal point of research for several mathematicians. See  \cite{KLPS,KLSP,LTWW, LTWW2, M,MMPS} for some interesting study in this area on various Banach spaces.  In \cite{LTWW}, this problem is addressed for $\mathcal{X}=\mathcal{Y}=L_p(\mu),$ where $1\leq p\leq \infty.$   In this article, we explore parallel or TEA pair preserving linear maps $T\in \mathcal{L}(\ell_p, \mathcal{Y}).$   For $1<p<\infty,$ $\ell_p$ being strictly convex, $(x,y)$ is a parallel pair in $\ell_p$ if and only if $x,y$ are scalar multiple of each other. Therefore all linear maps from $\ell_p$ preserve parallel pairs. Hence the study of such maps is particularly interesting for $p\in\{1,\infty\}.$ Similarly for $1<q<\infty,$ $(Tx,Ty)$ is a parallel pair in $\ell_q$ if and only if $\{Tx,Ty\}$ is linearly dependent. Therefore  for $1<q<\infty,$  $T\in \mathcal{L}(\mathcal{X},\ell_q)$ preserves parallel pairs if and only if $rank(T)\leq 1.$ 
 \\
   
   In Section \ref{sec-infinity}, we address this problem in the space $\mathcal{L}(\mathcal{X},\mathcal{Y}),$ where $\mathcal{X}\in \{\ell_\infty,\ell_\infty^n,c_0:n\geq 2\}.$  Theorems \ref{th-ncom} and \ref{th-nteag}  characterize the parallel and TEA pair preserving maps, respectively, that are not injective. We give a  special attention  to the case $\mathcal{Y}\in \{\ell_1,\ell_1^m:m\geq 2\}.$ Specifically for such range spaces,  we consider the following cases 
   \begin{itemize}
   	\item $\mathbb{F}\in \{\mathbb{R},\mathbb{C}\},$ if $\mathcal{X}\in \{\ell_\infty,\ell_\infty^n:n\geq 3\}$  and $\mathbb{F}=\mathbb{C},$ if $\mathcal{X}=\ell_\infty^2$ 
   	\item $\mathbb{F}=\mathbb{R},$ if $\mathcal{X}=\ell_\infty^2$ 
   \end{itemize}
   separately. Theorems \ref{th-inf} and  \ref{th-ncomp}  (resp. Corollary \ref{cor-ntea}) show that in the first case, $T$ preserves parallel (resp. TEA) pairs if and only if $rank(T)\leq 1$ (resp. $T=0$).  Corollary \ref{cor-n2} characterizes these maps for the second case, and exhibits the existence of nontrivial parallel and TEA pair preserves, in contrast to the previous case. Theorem \ref{th-n2p} characterizes the TEA pair preservers in $\mathcal{L}(\ell_\infty^2,\mathcal{Y}),$ where $\mathcal{Y}$ is strictly convex, and the scalars are real.\\
   In Section \ref{sec-l1}, we investigate the problem in the space $\mathcal{L}(\ell_1(J),\mathcal{Y}),$ where $J\subseteq \mathbb{N}.$ In this setting the problem is significantly difficult unlike the space $\mathcal{L}(\ell_\infty,\mathcal{Y}).$  The parallel pair  preserving maps in $\mathcal{L}(\ell_1(J),\mathcal{Y})$ exhibit a rich structure.  The section begins with a necessary condition for parallel pair preserving maps $T$ in Theorem \ref{th-rank}, which eventually reduces the problem to the case $|J|<\infty,$ when $1<rank(T)<\infty.$  Theorem \ref{th-inv} characterizes the invertible parallel pair preservers in $\mathcal{L}(\ell_1^n,\ell_\infty^n).$ Sufficient condition for parallel pair preserving maps in $\mathcal{L}(\ell_1^n,\ell_\infty^m)$ $(m,n\geq 2)$ is obtained. In particular, these maps in $\mathcal{L}(\ell_1^2,\ell_\infty^m)$ are characterized in Theorems \ref{th-com} and \ref{th-comr} for complex and real scalars, respectively.   Corollary \ref{cor-teainf} establishes that if either $\mathcal{Y}$ is strictly convex or $\mathcal{Y}=\ell_\infty^m$ and $\mathbb{F}=\mathbb{C},$ then   $(0\neq) ~T\in\mathcal{L}(\ell_1(J),\mathcal{Y})$ preserves TEA pairs  if and only if  the set  $\Lambda=\{i\in J:Te_i\neq 0\}$ is a singleton,  where   $e_i$ denotes the $i$-th standard coordinate vector of the corresponding space.   Throughout the article we follow this convention. 
   On the other hand, Theorem \ref{th-tear} characterizes TEA pair preservers in $\mathcal{L}(\ell_1^2,\ell_\infty^m)$ if the scalars are real, and proves the existence of such maps of rank 2 in contrast to the complex scalars. Finally Theorem \ref{th-ltea} characterizes the TEA pair preservers in $\mathcal{L}(\ell_1^2,\mathcal{Y}),$ where $\mathcal{Y}$ is strictly convex, and the scalars are real.\\

    For ease of comparison, in  Tables 1 and 2, we summarize the structural classifications of these maps established here. Notably, they emphasize the distinction between these maps over real and complex scalar fields when $\mathcal{X}\in \{\ell_\infty^2,\ell_1^2\}.$ 
   \begin{center}
   	Table 1 : Parallel pair preservers $T:\mathcal{X}\to \mathcal{Y}$
   	\begin{tabular}{ |c|c|c|c|} 
   		\hline
   		$\mathcal{X}$ & $\mathcal{Y}$ & $\mathbb{F}$ & Parallel pair preserver\\ 
   		\hline 
   		$\ell_\infty^2$ & $\ell_1$   & $\mathbb{R}$& $Te_j=(a_{1j},a_{2j},\ldots,a_{nj},\ldots)$, for $j=1,2,$ \\
   		&&& $a_{i2}\in \{a_{i1},-a_{i1}\},$ $i\in \mathbb{N}$\\
   		&&& or \\
   		&&&$rank(T)=1$\\
   		\hline
   		$\ell_\infty^2$ &  $\ell_1^m,~m\geq 2$  & $\mathbb{R}$& $Te_j=(a_{1j},a_{2j},\ldots,a_{mj})$, for $j=1,2,$ \\
   		&&& $a_{i2}\in \{a_{i1},-a_{i1}\},$ $1\leq i\leq m$\\
   		&&& or \\
   		&&& $rank(T)=1$\\
   		\hline
   		$\ell_\infty^2$ &  $\ell_1$ or $\ell_1^m,m\geq 2$  & $\mathbb{C}$& $rank(T)\leq 1$\\
   		\hline
   		$\ell_\infty^2$ &  Strictly convex   & $\mathbb{R}$ or $\mathbb{C}$& $rank(T)\leq 1$\\
   		\hline
   		$\ell_\infty$ or  &  Strictly convex or   & $\mathbb{R}$ or $\mathbb{C}$& $rank(T)\leq 1$\\
   		$\ell_\infty^n,$ $n\geq 3$	&$\ell_1$ or  $\ell_1^m,~m\geq 2$&&\\
   		\hline
   		$\ell_\infty$ or $c_0$ &  Finite-dimensional   & $\mathbb{R}$ or $\mathbb{C}$& $rank(T)\leq 1$\\
   		\hline
   		$\ell_\infty^n,$ $n\geq 3$&  $\dim(\mathcal{Y})<n$   & $\mathbb{R}$ or $\mathbb{C}$& $rank(T)\leq 1$\\
   		\hline
   	\end{tabular}
   	
   \end{center}
   \vspace{.3cm}
   \begin{center}
   	Table 2 : TEA pair preservers $T:\mathcal{X}\to \mathcal{Y}$
   	\begin{tabular}{ |c|c|c|c|} 
   		\hline
   		$\mathcal{X}$ & $\mathcal{Y}$ & $\mathbb{F}$ & TEA pair preserver\\ 
   		\hline
   		$\ell_\infty^2$ & $\ell_1$   & $\mathbb{R}$& $Te_j=(a_{1j},a_{2j},\ldots,a_{nj},\ldots)$, for $j=1,2,$ \\
   		&&& $a_{i2}\in \{a_{i1},-a_{i1}\},$ $i\in \mathbb{N}$\\
   		\hline
   		$\ell_\infty^2$ &  $\ell_1^m,~m\geq 2$  & $\mathbb{R}$& $Te_j=(a_{1j},a_{2j},\ldots,a_{mj})$, for $j=1,2,$ \\
   		&&& $a_{i2}\in \{a_{i1},-a_{i1}\},$ $1\leq i\leq m$\\
   		\hline
   		$\ell_\infty^2$ &  Strictly convex  & $\mathbb{R}$& $T(\cdot)=f(\cdot)y,$ where $y\in \mathcal{Y}, $ and \\
   		&&&$f\in \{(1,1),(1,-1)\}\subset \ell_1^2$ \\
   		\hline
   		$\ell_\infty^2$ &  Strictly convex or   & $\mathbb{C}$& $T=0$\\
   		&$\ell_1$ or  $\ell_1^m,~m\geq 2$&&\\
   		\hline
   		$\ell_\infty$ or  &  Strictly convex or   & $\mathbb{R}$ or $\mathbb{C}$& $T=0$\\
   		$\ell_\infty^n,$ $n\geq 3$	&$\ell_1$ or  $\ell_1^m,~m\geq 2$&&\\
   		\hline
   		$\ell_\infty$ or $c_0$ &  Finite-dimensional   & $\mathbb{R}$ or $\mathbb{C}$& $T=0$\\
   		\hline
   		$\ell_\infty^n,$ $n\geq 3$&  $\dim(\mathcal{Y})<n$   & $\mathbb{R}$ or $\mathbb{C}$& $T=0$\\
   		\hline
   		$\ell_1(J)$ &  Strictly convex   & $\mathbb{R}$ or $\mathbb{C}$ & $T=0$ or $\Lambda=\{i\in J:Te_i\neq 0\}$ is singleton \\
   		\hline
   			$\ell_1(J)$ & $ \ell_\infty^m,~m\geq 2$   & $\mathbb{C}$ & $T=0$ or $\Lambda=\{i\in J:Te_i\neq 0\}$ is singleton \\
   		\hline
   			$\ell_1^2$ & $ \ell_\infty^m,~m\geq 2$   & $\mathbb{R}$ & $T=0$ or $\Lambda=\{i\in J:Te_i\neq 0\}$ is singleton \\
   		&&& or upto permutation of rows, \\
   		&&&$T=\begin{bmatrix}
   			t_{11}&t_{11}&t_{31}&\ldots& t_{m1}\\
   				t_{12}&-t_{12}&t_{32}&\ldots& t_{m2}\\
   			\end{bmatrix}^T,$\\
   		&&& where $|t_{11}|\geq |t_{i1}|,$ $|t_{12}|\geq |t_{i2}|,$ $3\leq i\leq m$\\
   			\hline
   			$\ell_1^2$ & Strictly convex & $\mathbb{R}$ &$T(\cdot)=f(\cdot)y,$ where $y\in \mathcal{Y}, $ and \\
   			&&&$f\in \{e_1,e_2\}\subset \ell_\infty^2$  \\
   			\hline
   	\end{tabular}
   \end{center}
   
     \medskip

 Whenever appropriate, we identify $T\in \mathcal{L}(\ell_p^n,\ell_q^m)$ with its corresponding $m\times n$ matrix in the standard ordered basis. The symbols $\mathcal{X}^*,$ and $E_{\mathcal{X}}$ denote the dual space of $\mathcal{X},$ and the set of all extreme points of the closed unit ball of  $\mathcal{X},$ respectively. Suppose $J(x):=\{f\in \mathcal{X}^*:\|f\|=1,f(x)=\|x\|\}.$ 
We use the following characterization of parallel and TEA pairs.
\begin{itemize}
\item $(x,y)$ is a parallel (resp. TEA) pair in $\mathcal{X}$ if and only if there exists $f\in J(x)\cap J(\lambda y)$ for some $\lambda\in \mathbb{T}$ (resp. $f\in J(x)\cap J(y)$).
\end{itemize}
If $\dim(\mathcal{X})<\infty,$ then in the above characterization, we may further assume $f\in E_{\mathcal{X}^*}.$ In particular, if $x=(x_1,\ldots,x_n), y=(y_1,\ldots,y_n)\in \mathbb{F}^n,$ 
 then the following holds.
\begin{itemize}
\item $(x,y)$ is a	parallel (resp. TEA) pair in $\ell_1^n$ if and only if there exists $\lambda\in \mathbb{T}$ (resp. $\lambda=1$) such that $\lambda \overline{x_k}y_k\geq 0$ for all $1\leq k\leq n.$
\item $(x,y)$ is a	parallel (resp. TEA) pair in $\ell_\infty^n$ if and only if $\|x\|=|x_k|,\|y\|=|y_k|$ (resp. $\|x\|=|x_k|,\|y\|=|y_k|,$ and $\overline{x_k}y_k\geq 0$) for some $k\in \{1,\ldots,n\}.$ 
\end{itemize}
 For $x\in \{\ell_\infty,\ell_\infty^n,c_0\},$ if $\|x\|=|x_k|$ holds, where $x_k$ denotes the $k$-th coordinate of  $x,$ then we say that $x$ attains its norm at the $k$-th coordinate. Thus $(x,y)$ is a parallel pair in $\ell_\infty^n$ if and only if both $x$ and $y$ attain their norm at $k$-th coordinate for some $1\leq k\leq n.$ A non-zero vector $x\in \mathcal{X}$ is said to be smooth, if $J(x)$ is singleton.
The following lemma is used extensively in the sequel.

\begin{lemma}\label{lem-pdim}\cite[Lemma 2.3]{M} 
	Suppose $\mathcal{X}$ contains a smooth vector (in particular, $\mathcal{X}$ is separable). If $x\parallel y$ for all $x,y\in \mathcal{X},$ then $\dim(\mathcal{X})\leq 1.$	
\end{lemma}

\section{Parallel and TEA pair preservers in $\mathcal{L}(\ell_\infty,\mathcal{Y})$ }\label{sec-infinity}
The primary goal of this section is to characterize the parallel and TEA pair preservers in $\mathcal{L}(\mathcal{X},\mathcal{Y}),$  where  $\mathcal{X}\in \{\ell_\infty^n,\ell_\infty,c_0:n\geq 2\}.$  To achieve the goal, we first prove the following lemmas.

\begin{lemma}\label{lem-r1}
	Suppose $T\in \mathcal{L}(\mathcal{X},\mathcal{Y})$  preserves parallel pairs, where  $\mathcal{X}\in \{\ell_\infty^n,\ell_\infty,c_0:n\geq 2\}.$  Let $Tx=0$ for some non-zero vector $x=\sum_{i=1}^kx_ie_i\in \mathcal{X}.$ Then $\dim(T(V))\leq 1,$ where $$V=\{v\in \mathcal{X}:j\text{-th coordinate of } v \text{ is }0, 1\leq j\leq k\}.$$
	In particular, if $Te_i=0$ for some $i\geq 1,$ then $rank(T)\leq 1.$
\end{lemma}
\begin{proof}
 We claim that arbitrary two elements of $T(V)$ are parallel. Suppose $u,v\in V.$ Choose $\xi\in \mathbb{R}$ such that  $\xi \|x\|>\max\{\|u\|,\|v\|\}.$ If $\|x\|=|x_i|,$ then both $\xi x+u$ and $\xi  x+v$ attain their norm at the $i$-th coordinate. Therefore, 
\[(\xi x+u,\xi x+v) \text{ is a parallel pair in }\mathcal{X}\quad \Rightarrow \quad (Tu,Tv) \text{ is a parallel pair in }\mathcal{Y}.\]
Since each standard coordinate vector is smooth in $\mathcal{X},$ by Lemma \ref{lem-pdim}, $\dim(T(V))\leq 1.$ 
The next part of the proof follows from  $$range(T)=span\{Te_i\}\oplus T(W),$$ where $W=\{w\in \mathcal{X}:i\text{-th coordinate of } w \text{ is }0\},$ and the first part.
\end{proof}

\begin{lemma}\label{lem-nind}
		Suppose $T\in \mathcal{L}(\mathcal{X},\mathcal{Y})$  preserves parallel pairs, where  $\mathcal{X}\in\{\ell_\infty^n,\ell_\infty, c_0:n\geq 2\}.$ 
		 Assume $\ker(T)\neq \{0\}.$ Then   $\{Te_k,Te_j\}$ is linearly dependent, for some $ k<j, $ $e_k,e_j\in \mathcal{X}.$
\end{lemma}
\begin{proof}
Let $x\in \ker(T),$ and $x\neq 0.$	 Suppose $x_j$ is the $j$-th coordinate of $x.$ Assume $x_j=|x_j|e^{i\theta_j},$ if $x_j\neq 0.$ If $x$ has at most two non-zero coordinates, then it is trivial. So assume $x$ has at least three non-zero coordinates. \\

\textbf{Case a:} Suppose $\sup\{|x_j|:x_j\neq 0\}\neq \inf\{|x_j|:x_j\neq 0\}.$ Without loss of generality, assume $|x_1|>|x_2|>0.$ We show that $\{Te_1,Te_2\}$ is linearly dependent by proving that arbitrary two elements of $span\{Te_1,Te_2\}$ form a parallel pair in $\mathcal{Y}$, and then using Lemma \ref{lem-pdim}. Suppose $y,z\in span\{e_1,e_2\}.$ If $(y,z)$ is a parallel pair in $\mathcal{X},$ then so is $(Ty,Tz)$  in $\mathcal{Y}.$ Let $(y,z)$ be not a parallel pair in $\mathcal{X}.$ Assume $y=ae_1+be_2, z=ce_1+de_2.$ Clearly, either $|a|>|b|, |c|<|d|$ or $|a|<|b|,|c|>|d|.$ Without loss of generality, assume $|a|<|b|,|c|>|d|,$ and (since parallelism is homogeneous) $y=\alpha e_1+e_2, z=e_1+\beta e_2,$ where $\alpha=\frac{a}{b},\beta=\frac{d}{c},$ $|\alpha|<1,|\beta|<1.$ Let $\alpha=|\alpha|e^{i\theta},\beta=|\beta|e^{i\phi}.$\\
If $\|x\|=|x_1|,$ then for $\xi=|\xi|e^{i(\theta-\theta_1)},$ where $|\xi|>\frac{1-|\alpha|}{|x_1|-|x_2|},$ 
 $y+\xi x$ attains its norm at the first coordinate, which yields that $(y+\xi x,z)$ is a parallel pair in $\mathcal{X},$ and so is $(Ty,Tz)$ in $\mathcal{Y}.$\\
 If $\|x\|> |x_1|,$ then for $\xi>\frac{1}{\|x\|-|x_1|},$ $(y+\xi x,z+\xi x)$ is a parallel pair in $\mathcal{X},$ and so is $(Ty,Tz)$ in $\mathcal{Y}.$\\

\textbf{Case b:} Suppose $\sup\{|x_j|:x_j\neq 0\}= \inf\{|x_j|:x_j\neq 0\}.$   Then $|x_i|=|x_j|$ for all non-zero $x_i,x_j.$ Choose three non-zero coordinates. Without loss of generality, assume $x_i\neq 0$ for $i=1,2,3,$ that is, $|x_1|=|x_2|=|x_3|>0.$ We show that $\{Te_1,Te_2\}$ is linearly dependent. As in the previous case, suppose $y,z\in span\{e_1,e_2\},$ and  without loss of generality, assume $y=\alpha e_1+e_2, z=e_1+\beta e_2,$ where  $|\alpha|<1,|\beta|<1.$ Let $\alpha=|\alpha|e^{i\theta},\beta=|\beta|e^{i\phi}.$ Now we consider $\mathbb{F}=\mathbb{R},$ and $\mathbb{F}=\mathbb{C}$ separately. \\
\textbf{Case b1:} 
Suppose $\mathbb{F}=\mathbb{R}.$  Choose $$\xi=-|\xi|e^{-i\theta_1}, \text{ where }|\xi|>\frac{1}{|x_1|},\quad \text{ and }\quad  \eta =\frac{-1}{x_1}.$$ For $x_1=x_2,$ observe the following holds in $\mathcal{X}$:
\begin{itemize}
\item If $\alpha\geq 0,\beta\geq 0,$ then	$(y+\xi x,z+\xi x)$ is a parallel pair, since both $y+\xi x,z+\xi x$ attain their norm at the third coordinate.
\item If $\alpha\geq 0,\beta< 0,$ then $(y,z+\eta x)$ is a parallel pair.
\item If $\alpha<0,$ then $(y+\eta  x,z)$ is a parallel pair.
\end{itemize}
For $x_1=-x_2,$  observe the following holds in $\mathcal{X}$:
\begin{itemize}
	\item If $\alpha\geq 0,$ then	$(y-\eta x,z)$ is a parallel pair.
	\item If $\alpha< 0,\beta\geq  0,$ then $(y,z+\eta x)$ is a parallel pair.
	\item If $\alpha<0,\beta<0,$ then $(y-\xi x,z+\eta x)$ is a parallel pair,  since both $y-\xi x,z+\eta x$ attain their norm at the third coordinate.
\end{itemize}
Therefore, in each case $(Ty,Tz)$ is a parallel pair in $\mathcal{Y}.$

\noindent \textbf{Case b2:} 
Suppose $\mathbb{F}=\mathbb{C}.$
Choose $\xi=-|\xi|e^{-i\theta_2},$ where $|\xi|>\frac{1}{|x_2|}.$ Then 
\[|1+\xi x_2|=|\xi x_2|-1<|\xi x_1|-|\alpha|\leq |\alpha+\xi x_1|.\]
If $|\alpha+\xi x_1|\geq |\xi x_3|,$ then  $y+\xi x$ attains its norm at the first coordinate, which yields that $(y+\xi x,z)$ is a parallel pair in $\mathcal{X},$ and so is $(Ty,Tz)$ in $\mathcal{Y}.$ Suppose $|\alpha+\xi x_1|< |\xi x_3|.$  Clearly,  $y+\xi x$ attains its norm at the third coordinate. Our goal is now to  construct $\eta\in \mathbb{C}$ such that $z+\eta x$ also attains its norm at the third coordinate. This would imply $(y+\xi x,z+\eta x)$ is a parallel pair in $\mathcal{X},$ and so is $(Ty,Tz)$ in $\mathcal{Y}.$\\

If $\beta=0,$ then for $\eta=-|\eta|e^{-i\theta_1},$ where $|\eta|>\frac{1}{|x_1|},$ 
$z+\eta x$  attains its norm at the third coordinate.\\ 

Suppose $\beta\neq 0.$ In this case, the construction of $\eta$ depends on the argument of $\beta.$ Let $\phi_\mu=\theta_2-\theta_1-\mu,$ for some $0\leq \mu<\pi.$ Consider 
$$ \beta_\mu=|\beta|e^{i\phi_\mu},\text{ and}\quad z_\mu=e_1+\beta_\mu e_2.$$
Choose $\epsilon>0$ such that $0\leq \mu<\mu+\epsilon<\pi.$ Let $\psi=\frac{\pi}{2}-\theta_1+\epsilon.$ Note that \[\cos(\psi+\theta_1)<0, ~\text{ and } \cos(\psi+\theta_2-\phi_\mu)<0.\] 
Define 
\begin{equation}\label{eq-n001}
	 \eta =|\eta|e^{i\psi}, \text{ where } |\eta|>\max\Big\{\frac{-1}{2|x_1|\cos(\psi+\theta_1)}, \frac{-|\beta_\mu|}{2|x_2|\cos(\psi+\theta_2-\phi_\mu)}\Big\}.
	 \end{equation}
Now it is easy to check that 
\[|1+\eta x_1|\leq |\eta x_3|,\text{ and } |\beta_\mu+\eta x_2|\leq |\eta x_3|.\]
Therefore, $z_\mu+\eta x$ attains its norm at the third coordinate. \\
 Similarly, if $\phi_\mu=\theta_2-\theta_1-\mu,$ for some $\pi< \mu<2\pi,$ then 
 choose $$\psi\in \Big(\frac{5\pi}{2}-\theta_1-\mu,\frac{3\pi}{2}-\theta_1\Big),$$ and define $\eta$ as \eqref{eq-n001}. Then  $z_\mu+\eta x$ attains its norm at the third coordinate, and so $(y+\xi x, z_\mu+\eta x)$ is a parallel pair in $\mathcal{X}.$ Therefore, $(Ty,Tz_\mu)$ is a parallel pair in $\mathcal{Y}.$ Using a continuity argument, we get,  $(Ty,Tz_\mu)$ is a parallel pair in $\mathcal{Y},$ for all $\mu\in [0,2\pi).$ Therefore, $(Ty,Tz)$ is a parallel pair in $\mathcal{Y}.$\\
 Since arbitrary two elements of $span\{Te_1,Te_2\}$ form a parallel pair, by Lemma \ref{lem-pdim}, $\{Te_1,Te_2\}$ is linearly dependent.
\end{proof}

The following theorem characterizes the parallel pair preservers in $\mathcal{L}(\ell_\infty,\mathcal{Y})$ that are not injective.
\begin{theorem}\label{th-ncom}
	Suppose $T\in \mathcal{L}(\mathcal{X},\mathcal{Y}),$  where  $\mathcal{X}\in\{\ell_\infty^n,\ell_\infty,c_0:n\geq 2\}.$ Assume $\ker(T)\neq \{0\}.$ Then $T$ preserves parallel pairs  if and only if $rank(T)\leq 1.$
\end{theorem}
\begin{proof}
	The sufficient part is trivial. For the necessary part, assume $T$ preserves parallel pairs. 
	If $\mathcal{X}=\ell_\infty^2,$ then the result follows immediately from $\ker(T)\neq 0.$ Suppose $\mathcal{X}\neq \ell_\infty^2.$ 	Using Lemma \ref{lem-nind}, without loss of generality, assume $\{Te_1,Te_2\}$ is linearly dependent. So $Tx=0$ for some non-zero $x=x_1e_1+x_2e_2.$ By Lemma \ref{lem-r1}, $\dim(T(V))\leq 1,$ where
	$$V=\{v\in \mathcal{X}:j\text{-th coordinate of } v \text{ is }0, j=1,2\}.$$
	If  $Te_i=0$ for some $i\geq 1,$ then the result follows from Lemma \ref{lem-r1}. Assume  $Te_i\neq 0$ for all $i\geq 1.$ Then $x_1\neq 0,x_2\neq 0.$
	Note that $$range(T)=span\{Te_1,Te_2, Tv:v\in V\} =span\{Te_1,Te_3\}.$$ 
We show that arbitrary two elements of $span \{Te_1,Te_3\}$ are parallel. Let $y,z\in span\{e_1,e_3\}.$ If $(y,z)$ is a parallel pair in $\mathcal{X},$ then so is $(Ty,Tz)$ in $\mathcal{Y}.$ Suppose $(y,z)$ is not a parallel pair in $\mathcal{X}.$ As in the previous lemma, without loss of generality, we may assume that 
	\[y=\alpha e_1+e_3, \quad z=e_1+\beta e_3,\]
	where $|\alpha|<1,|\beta|<1.$ Suppose $\alpha=|\alpha|e^{i\theta},\beta=|\beta|e^{i\phi},$ and $x_j=|x_j|e^{i\theta_j}$ for $j=1,2.$\\ 
	If $|x_1|\geq |x_2|,$ then for $\xi=|\xi|e^{i(\theta-\theta_1)},$ where $|\xi|>\frac{1-|\alpha|}{|x_1|},$  $(y+\xi x,z)$ is a parallel pair in $\mathcal{X}.$
	If $|x_1|< |x_2|,$ then choose $\xi=|\xi|e^{i(\theta-\theta_1)},$ where $|\xi|>\max\Big\{\frac{|\alpha|}{|x_2|-|x_1|}, \frac{1}{|x_2|}\Big\},$ and $\eta=|\eta|e^{-i\theta_1},$ where $|\eta|>\Big\{\frac{|\beta|}{|x_2|},\frac{1}{|x_2|-|x_1|}\Big\}.$ Note that  $(y+\xi x,z+\eta x)$ is a parallel pair in $\mathcal{X}.$
Thus $(Ty,Tz)$ is a parallel pair in $\mathcal{Y},$ and so by Lemma \ref{lem-pdim}, $\{Te_1,Te_3\}$ is linearly dependent, proving that $rank(T)\leq 1.$ 
	\end{proof}

Next we show that, in particular,  if $\mathcal{Y}\in\{\ell_1,\ell_1^m:m\geq 2\},$ and either $\mathcal{X}\neq \ell_\infty^2$ or the field is complex, then the parallel pair preservers $ T\in\mathcal{L}(\mathcal{X},\mathcal{Y})$ are characterized by $rank(T)\leq 1.$ 
The following lemma is required for this purpose, which is of independent interest. 
Suppose $T:\mathcal{X}\to \mathcal{Y}\oplus_1 \mathcal{Z}.$ Let   $T_1:\mathcal{X}\to \mathcal{Y}, T_2:\mathcal{X}\to \mathcal{Z}$ be  the projections of $T$ onto the Banach spaces $\mathcal{Y}$ and $\mathcal{Z},$ respectively, that is,  $Tx=(T_1x,T_2x)$ for all $x\in \mathcal{X}.$ Then we write $T=(T_1,T_2).$
\begin{lemma}\label{lem-ngen}
	Suppose $T:\mathcal{X}\to \mathcal{Y}\oplus_1 \mathcal{Z}$ preserves parallel pairs, and $T=(T_1,T_2).$ Then $T_1,T_2$ preserves parallel pairs.
\end{lemma}
\begin{proof}
	Suppose $(x,w)$ is a parallel pair in $\mathcal{X}.$ Then $(Tx,Tw)$ is a parallel pair in $\mathcal{Y}\oplus_1 \mathcal{Z}.$ Suppose for $\lambda\in \mathbb{T},$ and $f\in (\mathcal{Y}\oplus_1 \mathcal{Z})^*=\mathcal{Y}^*\oplus_\infty \mathcal{Z}^*,$ 
	\[\|f\|=1,~f(Tx)=\|Tx\|, \text{ and } f(\lambda Tw)=\|Tw\|.\]	Let $f=(f_1,f_2),$ where $f_1\in \mathcal{Y}^*, f_2\in \mathcal{Z}^*.$ Then
	\[\|T_1x\|+\|T_2x\|=\|Tx\|=f(Tx)=f_1(T_1x)+f_2(T_2x)\leq \|T_1x\|+\|T_2x\|,	\]
	which implies that $f_i(T_ix)=\|T_ix\|,$ and so $\|f_i\|=1$ for $i=1,2.$ Similarly, $f_i(\lambda T_iw)=\|T_iw\|,$ for $i=1,2.$ Therefore $(T_ix,T_iw)$ is a parallel pair. Thus $T_i$ preserves parallel pairs, for $i=1,2.$
\end{proof}
We further need the  observation that if $x\in E_{\ell_\infty^n},$ then $(x,z)$ is a parallel pair in $\ell_\infty^n$ for all $z\in \ell_\infty^n.$ This also holds in the infinite-dimensional setting. Indeed suppose $x=(e^{i\phi_1},e^{i\phi_2},\ldots)\in E_{\ell_\infty},$ and  $z=(|z_1|e^{i\theta_1},|z_2|e^{i\theta_2},\ldots)\in \ell_\infty.$ If $\|z\|$ is attained at some coordinate, then clearly $(x,z)$ is a parallel pair in $\ell_\infty.$ Suppose $\|z\|$ is not attained at any coordinate. Without loss of generality, (if necessary passing through subsequences) assume that $$\|z\|=\lim_{n\to \infty}|z_n|, \quad\lim_{n\to\infty} e^{i\theta_n}=e^{i\theta}, \quad\lim_{n\to\infty} e^{i\phi_n}=e^{i\phi}.$$ Then for $\lambda=e^{i(\theta-\phi)},$ $\|z+\lambda x\|=\|z\|+\|x\|,$ consequently $(x,z)$ is a parallel pair in $\ell_\infty.$
Now we characterize the parallel  pair preservers when $\mathcal{X}\neq \ell_\infty^2.$
\begin{theorem}\label{th-inf}
		Suppose $T\in \mathcal{L}(\mathcal{X},\mathcal{Y})$, where  $\mathcal{X}\in\{\ell_\infty^n,\ell_\infty:n\geq 2\},$  and $\mathcal{Y}\in\{\ell_1^m,\ell_1:m\geq 2\}.$  
		\begin{enumerate}
			\item If $T$ preserves parallel pairs, then there exist subspaces $V$ and $W$ of $\mathcal{Y}$ such that  $range(T)\subseteq V\oplus_1 W.$ Moreover, if $T=(T_1,T_2),$ where $T_1:\mathcal{X}\to V,$ and $T_2:\mathcal{X}\to W$  are projections of $T$ onto $V$ and $W,$ respectively, then $rank(T_i)\leq 1,$ for $i=1,2.$
			\item If $\mathcal{X}\neq \ell_\infty^2,$ then $T$ preserves parallel pairs if and only if $rank(T)\leq 1.$
			\end{enumerate}
\end{theorem}
\begin{proof}
 $(1)$ Choose $x\in E_{\mathcal{X}}$ arbitrarily. For $k\geq 1,$ let $a_k=|a_k|e^{i\theta_k}$ be the $k$-th coordinate of $Tx.$ Suppose $$\Lambda=\{i\geq 1: a_i\neq 0 \},\text{ and }\Gamma=\{i\geq 1: a_i= 0 \}.$$ If $\Lambda=\emptyset,$ then $x\in \ker(T)$ and so by Theorem \ref{th-ncom}, $rank(T)\leq 1.$ Assume $\Lambda\neq \emptyset.$ Suppose $$V=span\{Tx\}, \quad W=\{y\in \mathcal{Y}:i\text{-th coordinate of }y \text{ is } 0, i\in \Lambda\}.$$ 
We claim that $range(T)\subseteq V\oplus_1 W.$ Let $z\in \mathcal{X}.$ Since $(x,z)$ is a parallel pair in $\mathcal{X},$ so is $(Tx,Tz)$ in $\mathcal{Y}.$ Suppose $b_k$ is the $k$-th coordinate of $Tz.$ Then there exists $\lambda\in \mathbb{T}$ such that $\lambda a_k\overline{b_k}\geq 0,$  that is, $\overline{\lambda}b_k= |b_k|e^{i\theta_k}$ for all $k\in \Lambda.$ If $b_k=0$ for all $k\in \Lambda,$ then clearly $Tz\in W.$ Suppose $b_j\neq 0$ for some $j\in \Lambda.$ Note that $|a_kb_j|=|a_jb_k|$ for all $k\in \Lambda.$ Indeed, suppose $\xi=|a_kb_j|-|a_jb_k|$ for some $k\in \Lambda.$ Let  
$$u=(|b_j|+|b_k|)x-(|a_j|+|a_k|)\overline{\lambda} z.$$ 
Suppose $c_i$ denotes the $i$-th coordinate of $Tu.$ 
Then  $c_j=-\xi e^{i\theta_j},$ and $c_k=\xi e^{i\theta_k}.$ Since $(x,u)$ is a parallel pair in $\mathcal{X},$ so is $(Tx,Tu)$ in $\mathcal{Y}.$ However, since $a_j\overline{c_j}=-\xi |a_j|,$ and $a_k\overline{c_k}=\xi |a_k|,$ if $\xi\neq 0,$ then there does not exist $\mu\in \mathbb{T}$ such that $\mu a_i\overline{c_i}\geq 0$ for all $i\geq 1,$ contradicting that $(Tx,Tu)$  is a parallel pair in $\mathcal{Y}.$ Thus $|a_kb_j|=|a_jb_k|,$ and so the $k$-th coordinate of $\frac{\overline{\lambda}}{|b_j|}Tz$ is $|\frac{b_k}{b_j}|e^{i\theta_k}=|\frac{a_k}{a_j}|e^{i\theta_k}$ for all $k\in \Lambda.$ Hence $\frac{\overline{\lambda}}{|b_j|}Tz=\frac{1}{|a_j|}Tx+y,$ for some $y\in W.$ This proves that $range(T)\subseteq V\oplus_1 W,$ and so $T:\mathcal{X}\to V\oplus_1 W $ preserves parallel pairs. Therefore by Lemma \ref{lem-ngen}, $T_1:\mathcal{X}\to V,$ and $T_2:\mathcal{X}\to W$ preserve parallel pairs, where $T_1,T_2$ are projections of $T$ onto $V$ and $W,$ respectively. Clearly $rank(T_1)\leq 1.$ On the other hand, since $T_2x=0,$ by Theorem \ref{th-ncom}, $rank(T_2)\leq 1.$\\

$(2)$ The sufficient part is trivial. For the necessary part, assume $T$ preserves parallel pairs. Consider $T_1,T_2, V,W$ as in $(1).$ Suppose $$T_1(\cdot)=f_1(\cdot)v, \text{ and }T_2(\cdot)=f_2(\cdot)w,$$ where $f_1,f_2\in \mathcal{X}^*,$ $v\in V,w\in W.$ Since $\ker(f_1)\cap \ker(f_2)\subseteq \ker(T),$  from Theorem \ref{th-ncom}, $rank(T)\leq 1.$ 
\end{proof}

Our next goal is to characterize  the parallel  pair preservers when $\mathcal{X}= \ell_\infty^2.$ In this case, the description of parallel pair preserving maps are different for real and complex scalars. We first deal with the complex scalars.
The following lemma is required for this.

\begin{lemma}\label{lem-ninv}
	Assume $\mathbb{F}=\mathbb{C}.$ 
	Suppose $T\in \mathcal{L}(\ell_\infty^2, \ell_1^2)$ preserves parallel pairs.  Then $T$ is not invertible.	
\end{lemma}
\begin{proof}
	Suppose for a contradiction that $T$ is invertible. Let $A=\{(1,1),(1,-1)\}.$   
	For all $x\in E_{\ell_\infty^2},$ and $y\in \ell_\infty^2,$ $(x,y)$ is a parallel pair in $\ell_\infty^2,$ and so is $(Tx,Ty)$ in $\ell_1^2.$ Since $range(T)=\ell_1^2,$ $(Tx,z)$ is a parallel pair in $\ell_1^2$ for all $z\in \ell_1^2.$ We claim that $\frac{Tx}{\|Tx\|}\in E_{\ell_1^2},$ where $x\in E_{\ell_\infty^2}.$ Without loss of generality, assume $\|Tx\|=1.$ Suppose for a contradiction $Tx\notin E_{\ell_1^2}.$ Then $Tx=(a_1,a_2),$ where  $a_1a_2\neq 0.$ For $u=(a_1,-a_2)\in \ell_1^2,$ $(Tx,u)$ is a parallel pair in $\ell_1^2.$  This contradiction proves the claim. Therefore, $\{\frac{T(1,1)}{\|T(1,1)\|},\frac{T(1,-1)}{\|T(1,-1)\|}\}$ is a linearly independent subset of $E_{\ell_1^2}=\{\mu e_i:i=1,2,\mu\in \mathbb{T}\}.$ 
	Since $e=(1,i)\in E_{\ell_\infty^2},$  $\{Te,Tx\}$ must be linearly dependent for some $x\in\{(1,1),(1,-1)\},$ whereas $\{e,x\}$ is linearly independent. This contradicts that $T$ is invertible.
\end{proof}

 The next theorem characterizes the parallel  pair preservers, when  $\mathcal{X}=\ell_\infty^2,$ and the scalars are complex.
\begin{theorem}\label{th-ncomp}
	Assume $\mathbb{F}=\mathbb{C},$ and $\mathcal{Y}\in\{\ell_1,\ell_1^m\}.$ 
	Then  $T\in\mathcal{L}(\ell_\infty^2, \mathcal{Y})$ preserves parallel pairs if and only $rank(T)\leq 1.$
\end{theorem}
\begin{proof}
	If $rank(T)\leq 1,$ then clearly $T$ preserves parallel pairs. Conversely assume that $T$ preserves parallel pairs. If $\mathcal{Y}=\ell_1^2,$ then by Lemma \ref{lem-ninv}, $\ker(T)\neq\{0\},$ and so  by Theorem \ref{th-ncom}, $rank(T)\leq 1.$ \\
	Now consider $\mathcal{Y}\neq \ell_1^2.$  Following the notation of Theorem \ref{th-inf} part $(1)$, $rank(T_i)\leq 1$ for $i=1,2.$ Thus   there exist $z_1\in V,z_2\in W, f_1,f_2\in (\ell_\infty^2)^*$ such that $$T_i(x)=f_i(x)z_i \text{ for all }x\in \ell_\infty^2, \text{ and }i=1,2.$$ If either $f_1$ or $f_2$ is zero, then $rank(T)\leq 1.$ Assume 
	\[f_1=(a,b), f_2=(c,d)\in \ell_1^2\setminus\{(0,0)\}.\]
	Choose $x\in E_{\ell_\infty^2}$ such that $f_1(x)\neq 0, f_2(x)\neq 0.$ Then for all $y\in \ell_\infty^2,$  $(x,y)$ is a parallel pair in $\ell_\infty^2,$ and so $(T_ix,T_iy)$ is a parallel pair 
	for $i=1,2.$ Proceeding similarly as Lemma \ref{lem-ngen}, we have $\lambda\in \mathbb{T},$ and $g_1\in V^*, g_2\in W^*$ such that 
	\[\|g_i\|=1,~g_i(T_ix)=\|T_ix\|, \text{ and } g_i(\lambda T_iy)=\|T_iy\|, \text{ for } i=1,2.\] Now, 
	\[\|T_i(x)\|=g_i(T_ix)\Rightarrow |f_i(x)|\|z_i\|=f_i(x)g_i(z_i),\]
	and so there exist $\mu_i\in \mathbb{T}$ such that $g_i(z_i)=\mu_i\|z_i\|,$ and $f_i(x)=\overline{\mu_i}|f_i(x)|.$
	Similarly, $$ g_i(\lambda T_iy)=\|T_iy\|\Rightarrow \lambda f_i(y)g_i(z_i)=|f_i(y)|\|z_i\|\Rightarrow \lambda f_i(y)=\overline{\mu_i}|f_i(y)|.$$ Therefore, 
	\begin{equation}\label{eq-nr1}
		\mu f_1(y)\overline{f_2(y)}=|f_1(y)f_2(y)|\geq 0, \text{ for all } y\in \ell_\infty^2, \text{ where } \mu=\mu_1\overline{\mu_2}. 
	\end{equation}
In \eqref{eq-nr1}, putting $y=e_i,$ $i=1,2,$ we get $\mu a\overline{c}\geq 0,$ and $\mu b\overline{d}\geq0.$
Moreover, for all $y=(1,t)\in \ell_\infty^2,$
\begin{eqnarray*}
	\mu f_1(y)\overline{f_2(y)}\geq 0 &\Rightarrow & \mu\Big(a\overline{c}+|t|^2b\overline{d}+tb\overline{c}+a\overline{dt}\Big)\geq0,
\end{eqnarray*}
	which implies in particular  for $t\in \{1,i\},$
	\[\Im(\mu b\overline{c}+\mu a\overline{d})=\Re(\mu b\overline{c}-\mu a\overline{d})=0\quad\Rightarrow \quad \mu b\overline{c}=\overline{\mu a}d.\]
	If $abcd=0,$ then it is straightforward to check that $\{f_1,f_2\}$ is linearly dependent.  If $abcd\neq 0,$ suppose $\mu a=|a|e^{i\theta},$ and $c=|c|e^{i\theta},$ since $\mu a\overline{c}\geq 0.$ Then $f_2=\frac{\mu \overline{c}}{\overline{\mu a}}f_1,$
	which implies that $\{f_1,f_2\}$ is linearly dependent. In each case, without loss of generality, assume $f_1=\alpha f_2.$ Then $T(x)=f_2(x)(\alpha z_1,z_2)$ for all $x\in \ell_\infty^2,$ and so $rank(T)\leq 1.$
\end{proof}

Now we deal with the real scalars. For this purpose we require the characterization of parallel (or TEA) pair preservers in $\mathcal{L}(\ell_1^n,\ell_1^m),$ which can be found for $n=m$ in \cite[Theorem 2.2]{LTWW}. Using this,  it is easy to observe that the same characterization also holds for $n\neq m.$ For the sake of completeness, we include a proof here.
\begin{theorem}\label{th-li}
Suppose $T\in\mathcal{L}(\ell_1^n,\ell_1^m).$  Then 
\begin{enumerate}
\item $T$ preserves TEA pairs if and only if each row of $T$ has at most one non-zero entry.
\item 	$T$ preserves parallel pairs if and only if each row of $T$ has at most one non-zero entry or $rank(T)=1$.
\end{enumerate}	
\end{theorem} 
\begin{proof}
For $n=m,$ the result follows from 	\cite[Theorem 2.2]{LTWW}. Suppose $n<m.$ Define $R:\ell_1^m\to \ell_1^m$ as follows:
\[Re_i=\begin{cases}
	Te_i, &\text{ if } 1\leq i\leq n\\
	0,&\text{ if } n+1\leq i\leq m
	\end{cases}.\] 
Observe that each row of $T$ has at most one non-zero entry if and only if the same holds for $R.$ 
We claim that $R$ preserves parallel (resp. TEA) pairs if and only if $T$ preserves parallel (resp. TEA) pairs. This claim completes the proof.\\
 Suppose $R$ preserves parallel (resp. TEA) pairs.  Let $(x,y)$ be a parallel (resp. TEA) pair in $\ell_1^n.$ Then $(u,v)$ is a parallel (resp. TEA) pair in $\ell_1^m,$ where $u=(x,0), v=(y,0).$ Consequently, $(Ru,Rv)=(Tx,Ty)$ is a parallel (resp. TEA) pair in $\ell_1^m.$ Thus $T$ preserves parallel (resp. TEA) pairs. Conversely, assume $T$ preserves parallel (resp. TEA) pairs. Let $(z,w)$ be a parallel (resp. TEA) pair in $\ell_1^m.$ Suppose $z=(z_1,z_2),$ and $w=(w_1,w_2),$ where $z_1,w_1\in \ell_1^n,$ and $z_2,w_2\in \ell_1^{m-n}$ Clearly, $(z_1,w_1)$ is a parallel (resp. TEA) pair in $\ell_1^n,$ and so is $(Tz_1,Tw_1)$  in $\ell_1^m.$ Since $Rz=Tz_1,$ and $Rw=Tw_1,$  $(Rz,Rw)$ is a parallel  (resp. TEA) pair in $\ell_1^m.$ Thus $R$ preserves parallel  (resp. TEA) pairs. \\
 
 Similarly, for $n>m,$ defining  $P:\ell_1^n\to \ell_1^n$ as 
 $Pe_i=(Te_i,0) \text{ for } 1\leq i\leq n,$ and proceeding similarly, we get the result.
\end{proof}
Proceeding  as Theorem \ref{th-li} and using \cite[Theorem 3.5]{LTWW}, we get the next characterization of parallel (or TEA) pair preservers in $\mathcal{L}(\ell_1^n,\ell_1).$ The proof is omitted to avoid monotony.

\begin{theorem}\label{th-linf}
	Suppose $T\in \mathcal{L}(\ell_1^n,\ell_1),$ and $Te_j=(a_{1j},a_{2j},\ldots)$ for all $1\leq j\leq n.$ Then 
	\begin{enumerate}
		\item $T$ preserves TEA pairs if and only if for each $i\in \mathbb{N},$ there is at most one one $j\in \{1,\ldots,n\}$ such that $a_{ij}\neq 0$.
		\item 	$T$ preserves parallel pairs if and only if either $T$ preserves TEA pairs  or $rank(T)=1$.
	\end{enumerate}		
	\end{theorem}

 Theorems \ref{th-li} and \ref{th-linf} directly yield the following characterization  of parallel and TEA pair preserves in $\mathcal{L}(\ell_\infty^2,\ell_1^m),$ where the scalar field is real. 
\begin{cor}\label{cor-n2}
	Suppose $T\in\mathcal{L}(\ell_\infty^2,\mathcal{Y}),$ where $\mathcal{Y}\in \{\ell_1,\ell_1^m:m\geq 2\}$, and $\mathbb{F}=\mathbb{R}.$ Let the $i$-th coordinate of $Te_j$ be $a_{ij}$ for  $j=1,2.$ Then 
	\begin{enumerate}
		\item $T$ preserves TEA pairs if and only if $a_{i2}\in \{a_{i1},-a_{i1}\}$ for all $i\geq 1.$
		\item 	$T$ preserves parallel pairs if and only if $a_{i2}\in \{a_{i1},-a_{i1}\}$ for all $i\geq 1$ or $rank(T)=1$.
	\end{enumerate}	
\end{cor}
\begin{proof}
The linear map $S:\ell_1^2\to \ell_\infty^2$ defined by 
$Se_1=(1,1), Se_2=(1,-1)$ is an isometry. Note that $T$ preserves parallel (resp. TEA) pairs if and only if $TS:\ell_1^2\to\mathcal{Y}$ preserves  parallel (resp. TEA) pairs. 
Consequently, the conclusion follows immediately from Theorems \ref{th-li} and  \ref{th-linf}. 
\end{proof}

Our final goal in this section is to characterize the TEA pair preservers for the remaining cases, that is, assuming either $\mathcal{X}\neq \ell_\infty^2$ or $\mathbb{F}\neq \mathbb{R}.$
\begin{theorem}\label{th-nteag}
	Suppose $T\in\mathcal{L}(\mathcal{X}, \mathcal{Y}),$ where either of the following holds:
	\begin{itemize}
		\item $\mathcal{X}\in \{\ell_\infty^n,\ell_\infty,c_0:n\geq 3\}.$ 
		\item $\mathbb{F}=\mathbb{C},$ $\mathcal{X}=\ell_\infty^2.$
	\end{itemize}
	Let $\ker(T)\neq \{0\}.$ Then  $T$ preserves TEA pairs if and only $T=0.$
\end{theorem}
\begin{proof}
	Suppose $T$ preserves TEA pairs. Then $T$ preserves parallel pairs. Since $\ker(T)\neq \{0\},$ from Theorem \ref{th-ncom}  it follows that $rank(T)\leq1.$  So there exist $f\in \mathcal{X}^*,$ and $z\in \mathcal{Y}$ such that $T(x)=f(x)z$ for all $x\in \mathcal{X}.$\\
	
	\noindent First assume  $\mathcal{X}\in \{\ell_\infty^n,\ell_\infty,c_0:n\geq 3\}.$ Then $\ker(f)\cap V\neq \{0\},$ where $V=span\{e_1,e_2\}.$ Choose a non-zero vector $x\in \ker(f)\cap V.$ Then $Tx=0.$ Suppose 
	$$W=\{w\in \mathcal{X}:i\text{-th coordinate of }w \text{  is }0, i=1,2 \}.$$ 
	For $w\in W$ and $\zeta>\frac{\|w\|}{\|x\|},$  $(\zeta x+w,\zeta x-w)$ is a TEA pair in $\mathcal{X},$ and so $(Tw,-Tw)$ is a TEA pair in $\mathcal{Y},$ that is,
	\[0=\|Tw-Tw\|=\|Tw\|+\|-Tw\|\Rightarrow Tw=0.\] Let $i\in\{1,2\}.$
	Then for $\eta>1,$ since $(\eta e_3+e_i,\eta e_3-e_i)$ is a TEA pair in $\mathcal{X}$ and $e_3\in W,$ so $(Te_i,-Te_i)$ is a TEA pair in $\mathcal{Y}.$ Proceeding similarly, we get $Te_i=0.$ Thus from $\mathcal{X}=V\oplus W,$ it follows that $T$ must be the zero operator.\\
	
	\noindent Now consider $\mathcal{X}=\ell_\infty^2,$ and $\mathbb{F}=\mathbb{C}.$  Choose a non-zero vector $x=(x_1,x_2)\in \ker(f).$ If $x_1=0,$ then $e_2\in \ker(f)=\ker(T).$ Then as previous, we get $Te_1=0,$ and so $T=0.$ Similarly, if $x_2=0,$ then $T=0.$ Thus it only remains to prove that $x_j=0$ for some $j\in \{1,2\}.$  Suppose for a contradiction that $x_j=|x_j|e^{i\theta_j}\neq 0,$ for $j=1,2.$ Without loss of generality, assume $|x_1|\geq |x_2|.$ Choose $y=(1,e^{i\theta}),$ where $\theta=\theta_2-\theta_1,$ and $\xi =|\xi|e^{-i\theta_1},$ where $|\xi|>\frac{1}{|x_2|}.$ Then $$\xi x+y=\Big(|\xi x_1|+1, (|\xi x_2|+1)e^{i\theta}\Big),\text{ and  }\xi x-y=\Big(|\xi x_1|-1, (|\xi x_2|-1)e^{i\theta}\Big),$$
	which implies that $\Big(\xi x+y,\xi x-y\Big)$ is a TEA pair in $\ell_\infty^2,$ and so $(Ty,-Ty)$ is a TEA pair in $\mathcal{Y},$ proving that $Ty=0,$ that is, $Te_1+e^{i\theta}Te_2=0.$ Moreover, $\Big((1,ie^{i\theta}),(1,0)\Big)$ is a TEA pair in $\ell_\infty^2,$ and so $\Big(T(1,ie^{i\theta}),T(1,0)\Big)$ is a TEA pair in $\mathcal{Y}.$ Therefore,
	\begin{eqnarray*}
		\|T(1,ie^{i\theta})+T(1,0)\|&=&\|T(1,ie^{i\theta})\|+\|T(1,0)\|\\
		\Rightarrow \|2Te_1+ie^{i\theta}Te_2\|&=&\|Te_1+ie^{i\theta}Te_2\|+\|Te_1\|\\
		\Rightarrow \|2Te_1-iTe_1\|&=&\|Te_1-iTe_1\|+\|Te_1\|, ~(\text{since } Te_1+e^{i\theta}Te_2=0) \\
		\Rightarrow |2-i|&=&|1-i|+1,
	\end{eqnarray*}
	which is clearly a contradiction.
\end{proof}

In the following corollary, we observe that in certain spaces there are no non-zero TEA pair preservers. 
\begin{cor}\label{cor-ntea}
	Suppose $\mathcal{Z}$ is a strictly convex Banach space, $\mathcal{Y}\in \{\ell_1^m,\ell_1, \mathcal{Z}:m\geq 2\},$ and either of the following holds:
	\begin{itemize}
		\item $\mathcal{X}\in \{\ell_\infty^n,\ell_\infty:n\geq 3\}.$ 
		\item $\mathbb{F}=\mathbb{C},$ $\mathcal{X}=\ell_\infty^2.$
	\end{itemize}
	Then  $T\in\mathcal{L}(\mathcal{X}, \mathcal{Y})$ preserves TEA pairs if and only $T=0.$
\end{cor}
\begin{proof}
	Suppose $T$ preserves TEA pairs. Then $T$ preserves parallel pairs. If $\mathcal{Y}$ is strictly convex, then $rank(T)\leq1.$ If $\mathcal{Y}\in \{\ell_1^m,\ell_1:m\geq 2\},$ then from Theorems \ref{th-inf} and \ref{th-ncomp} it follows that $rank(T)\leq1.$  In each case $\ker(T)\neq \{0\}.$ Thus the result follows from Theorem \ref{th-nteag}.
\end{proof}

Finally we end this section with the following characterization of TEA pair preserves in $ \mathcal{L}(\ell_\infty^2,\mathcal{Y}),$ where $\mathcal{Y}$ is strictly convex, and the scalars are real.
\begin{theorem}\label{th-n2p}
	Suppose $T\in \mathcal{L}(\ell_\infty^2,\mathcal{Y}),$  where $\mathcal{Y}$ is strictly convex, 
	and $\mathbb{F}=\mathbb{R}.$ Then $T$ preserves TEA pairs if and only if there exists $z\in \mathcal{Y}$ such that $Tx=f(x)z$ for all $x\in \ell_\infty^2,$ where $f\in \{(1,1),(1,-1)\}\subset \ell_1^2.$ 
\end{theorem}
\begin{proof}
For the sufficient part, assume 	$Tx=f(x)z$ for all $x\in \ell_\infty^2,$ where $f=(1,1).$ Let $\Big((x_1,x_2),(y_1,y_2\Big)$ be a TEA pair in $\ell_\infty^2.$ 
Suppose $$\|(x_1,x_2)\|=|x_1|\geq |x_2|,~\|(y_1,y_2)\|=|y_1|\geq |y_2|, \text{ and }x_1y_1\geq 0.$$ So 
\begin{eqnarray*}
\|T(x_1,x_2)+T(y_1,y_2)\|
&=&	|(x_1+x_2)+(y_1+y_2)|\|z\|\\
	&=&\|T(x_1,x_2)\|+\|T(y_1,y_2)\|,
\end{eqnarray*}
which implies that $T$ preserves TEA pairs. Similarly if $f=(1,-1),$ then  $T$ preserves TEA pairs.\\
For the necessary part, assume $T$ preserves TEA pairs. Then $T$ also preserves parallel pairs, and so $rank(T)\leq 1.$ If $T=0,$ then we are done.  Assume  $T\neq 0,$ and $T(\cdot)=f(\cdot)z$ for some non-zero $f=(a,b)\in \ell_1^2,$ and $(0\neq)~z\in \mathcal{Y}.$ Without loss of generality, assume $\frac{|a|}{|b|}\leq 1.$ Since $\Big((1,1),(1,-1)\Big)$ is a TEA pair in $\ell_\infty^2,$ so is $\Big(T(1,1),T(1,-1)\Big)$ in $\mathcal{Y}.$ Therefore,
\[\Big(|a+b|+|a-b|\Big)\|z\|=\|T(1,1)\|+\|T(1,-1)\|=\|T(1,1)+T(1,-1)\|=|2a|\|z\|,\]
which implies that $|a|=|b|.$ Thus either $f=a(1,1)$ or $f=a(1,-1),$ and so $Tx=h(x)az,$ where $h\in \{(1,1),(1,-1)\},$ and $az\in \mathcal{Y}.$  
\end{proof}

\section{Parallel and TEA pair preservers in $\mathcal{L}(\ell_1(J),\mathcal{Y})$} \label{sec-l1}

The primary goal of this section is to explore the parallel and TEA pair preservers in $\mathcal{L}(\ell_1(J),\mathcal{Y}).$ 
Let us begin with a necessary condition for parallel pair preserving maps.

\begin{theorem}\label{th-rank}
	Suppose $(0\neq)~T\in \mathcal{L}(\ell_1(J),\mathcal{Y})$  preserves parallel pairs.  If $\ker(T)\neq \{0\},$ then  either $rank(T)=1$ or $Te_i=0$ for some $i\in J.$	
\end{theorem}
\begin{proof}
	 Assume $Te_i\neq 0$ for all $i\in J.$ Choose a non-zero vector  $x\in \ker(T).$ Suppose $x=\sum_{i\in J}x_ie_i.$\\
	\textbf{Step 1:} $\{Te_k,Te_j\}$ is linearly dependent for some $k,j\in J.$\\
	This holds trivially if $x$ has at most two non-zero coordinates.  Assume $x$ has at least three non-zero coordinates.  Without loss of generality, suppose $x_i\neq 0$ for $i\in\{1,2,3\}\subseteq J.$ We  show that $Te_1,Te_2$ are linearly dependent. Let $y,z\in span\{e_1,e_2\}.$ If $(y,z)$ is a parallel pair in $\ell_1(J),$ then $(Ty,Tz)$ is a parallel pair in $\mathcal{Y}.$ Suppose $(y,z)$ is not a parallel pair in $\ell_1(J).$ Let $y=\alpha e_1+\beta e_2,~z=\gamma e_1+\delta e_2,$ for some scalars $\alpha,\beta,\gamma,\delta.$ Then $\alpha\beta\gamma\delta\neq 0.$ Since parallelism is homogeneous, $(\frac{\alpha}{\beta}e_1+e_2,\frac{\gamma}{\delta}e_1+e_2)$ is not a parallel pair in $\ell_1(J).$ Hence, either $\frac{\alpha}{\beta}\neq \frac{x_1}{x_2}$ or $\frac{\gamma}{\delta}\neq  \frac{x_1}{x_2}.$ Without loss of generality, assume that $\frac{\alpha}{\beta}\neq  \frac{x_1}{x_2}.$ Choose $\mu=\frac{\frac{x_1}{x_2}-\frac{\gamma}{\delta}}{\frac{x_1}{x_2}-\frac{\alpha}{\beta}},$ and $\lambda=\frac{1-\mu}{x_2}.$  Then 
	\[\mu(\frac{\alpha}{\beta}e_1+e_2)+\lambda x=\frac{\gamma}{\delta}e_1+e_2+u,\quad  \text{ where } u=\sum_{i\in J\setminus\{1,2\}}\lambda x_ie_i.\]
	Since $\Big(\frac{\gamma}{\delta}e_1+e_2+u,\frac{\gamma}{\delta}e_1+e_2\Big)$ is a parallel pair in $\ell_1(J),$ so is $\Big(\mu(\frac{\alpha}{\beta}e_1+e_2)+\lambda x,\frac{\gamma}{\delta}e_1+e_2\Big).$ Therefore, $\Big(\mu T(\frac{\alpha}{\beta}e_1+e_2)+\lambda Tx,T(\frac{\gamma}{\delta}e_1+e_2)\Big)$ is a parallel pair in $ \mathcal{Y}.$ Consequently, using $Tx=0,$ and homogeneity, we get  $\Big( T(\alpha e_1+\beta e_2),T(\gamma e_1+\delta e_2)\Big)=(Ty,Tz)$ is a parallel pair in $ \mathcal{Y}.$ Thus arbitrary two elements of $span\{Te_1,Te_2\}$ are parallel, which implies that $\dim(span\{Te_1,Te_2\})\leq 1.$\\
	\textbf{Step 2:} $rank(T)=1.$ \\
	By Step 1, without loss of generality, assume that $\{Te_1,Te_2\}$ is linearly dependent. Suppose  $aTe_1+bTe_2=0$ for some non-zero scalars $a,b.$ Let $w=ae_1+be_2.$
	Consider $i\in J\setminus\{1,2\}.$ Suppose $y,z\in span\{e_1,e_i\},$ and $y=\alpha e_1+\beta e_i, z=\gamma e_1+\delta e_i.$ If $(y,z)$ is a parallel pair in $\ell_1(J),$ then so  is $(Ty,Tz)$ in $\mathcal{Y}.$ Suppose $(y,z)$ is not a parallel pair in $\ell_1(J),$ that is, $\alpha\beta\gamma\delta\neq 0,$ and  $(\frac{1}{\beta}y,\frac{1}{\delta}z)$ is not a parallel pair in $\ell_1(J).$ So $\frac{\alpha}{\beta}\neq \frac{\gamma}{\delta}.$ For $\eta=\frac{\frac{\gamma}{\delta}-\frac{\alpha}{\beta}}{a},$ $$\frac{1}{\beta}y+\eta w=\frac{1}{\delta}z+\eta be_2.$$ Since $\Big(\frac{1}{\delta}z+\eta be_2, \frac{1}{\delta}z\Big)$ is a parallel pair in $\ell_1(J),$ so is $\Big(\frac{1}{\beta}y+\eta w,\frac{1}{\delta}z\Big).$ Thus $\Big(\frac{1}{\beta}Ty+\eta Tw,\frac{1}{\delta}Tz\Big)=\Big(\frac{1}{\beta}Ty,\frac{1}{\delta}Tz\Big)$ is a parallel pair in $\mathcal{Y},$ so is $(Ty,Tz).$ Since arbitrary two elements of $span\{Te_1,Te_i\}$ are parallel, $\dim(span\{Te_1,Te_i\})\leq 1.$ Thus $Te_i\in span\{Te_1\}$ for all $i\in J\setminus \{1\},$  and so $rank(T)=1.$ 
\end{proof}

The following corollary is an immediate consequence of the last theorem.
\begin{cor}\label{cor-fr}
Suppose $T\in \mathcal{L}(\ell_1(J),\mathcal{Y})$ preserves parallel pairs, and $rank(T)=r,$ where $1<r<\infty.$ Then $|\Lambda|=r,$ where $\Lambda=\{i\in J:Te_i\neq 0\}.$ 
\end{cor}
\begin{proof}
	Define $S:\ell_1(\Lambda)\to \mathcal{Y}$ by 
	$$Sx=\sum_{i\in \Lambda}x_i Te_i, \text{ where }x=\sum_{i\in \Lambda}x_ie_i\in \ell_1(\Lambda).$$ 
	Then $S$ preserves parallel pairs. Note that $\ker(S)=\{0\}.$ For otherwise, since $Se_i\neq 0$ for all $i\in \Lambda,$ Theorem \ref{th-rank} implies that $rank(S)=1.$ Thus from
	$$1=\dim span\{Se_i:i\in \Lambda\}=\dim span\{Te_i:i\in \Lambda\},$$
	it follows that $rank(T)=1,$ a contradiction. Now $\ker(S)=\{0\}$ implies that $\{Se_i:i\in \Lambda\}=\{Te_i:i\in \Lambda\}$ is linearly independent. Since $rank(T)<\infty,$  $\Lambda$ is a finite set. Therefore $range(T)=span\{Te_i:i\in \Lambda\},$ and so $rank(T)=|\Lambda|.$
\end{proof}

From the last corollary it follows that if $T\in \mathcal{L}(\ell_1(J),\mathcal{Y})$ preserves parallel pairs,  and $1<rank(T)<\infty,$ then $\Lambda=\{i\in J:Te_i\neq 0\}$ is a finite set.  Moreover, $T$ preserves parallel (or TEA) pairs if and only if $S$ preserves parallel (or TEA) pairs. Therefore to explore parallel (or TEA) pair preserving finite rank  maps in $\mathcal{L}(\ell_1(J),\mathcal{Y}),$ it is enough to explore these maps in $\mathcal{L}(\ell_1^n,\mathcal{Y}).$ In this direction, we first assume $\mathcal{Y}=\ell_\infty^m,$ $m\geq 2.$ Before proceeding further, we would like to emphasize that the problem of characterizing parallel and TEA pair preservers in $\mathcal{L}(\ell_1^n,\ell_\infty^m)$ is notably challenging. Indeed we will show the rich structures possessed by such maps, specially in $\mathcal{L}(\ell_1^2,\ell_\infty^m).$
Let us begin by showing that if $T\in \mathcal{L}(\ell_1^n,\ell_\infty^n)$ is an invertible parallel pair preserver, then all the entries of an arbitrary column of $T$ have same magnitude.
\begin{lemma}\label{lem-inv}
	Let	$T\in \mathcal{L}(\ell_1^n,\ell_\infty^n)$ be invertible, where $n\geq 2.$ Suppose $T$ preserves parallel pairs and $T=(t_{ij})_{i,j=1}^n.$ Then for each $1\leq j\leq n,$ there exists $k_j>0$ such that $|t_{ij}|=k_j$ for all $1\leq i\leq n.$
\end{lemma}
\begin{proof}
	Let $1\leq j\leq n.$ Choose $1\leq i\leq n$ arbitrarily. Suppose $X_i$ is the $(n-1)\times n$  submatrix of $T$ obtained by removing the $i$-th row of $T.$ Choose a non-zero vector $x\in \ell_1^n$ such that $X_ix=0.$ Then for $1\leq k\leq n, ~k\neq i,$ the $k$-th coordinate of  $Tx$ is zero. For $1\leq j\leq n,$ $(e_j,x)$ is a parallel pair in $\ell_1^n,$ and so is $(Te_j,Tx)$ in $\ell_\infty^n.$ Since $Tx\neq 0,$  $Tx$ attains its norm only at the $i$-th coordinate. Consequently, $Te_j$ also attains its norm at $i$-th coordinate, yielding that $|t_{ij}|\geq |t_{kj}|$ for all $1\leq k\leq n.$ Since this is true for arbitrary $i,$ $|t_{ij}|$ is constant for all $1\leq i\leq n .$ 
\end{proof}

The following theorem is a characterization of the invertible linear maps in $\mathcal{L}(\ell_1^n,\ell_\infty^n)$ that preserve parallel pairs.

\begin{theorem}\label{th-inv}
	Suppose $T\in \mathcal{L}(\ell_1^n, \ell_\infty^n)$  is an invertible linear map, where $n\geq 2.$ Then the following are equivalent.
	\begin{enumerate}
		\item $T$  preserves parallel pairs.
		\item There exists an $n\times n$ invertible diagonal matrix $D$ such that $TD=(e^{i\xi_{rj}})_{r,j=1}^n,$ and for each $\Theta=(\theta_1,\theta_2,\ldots,\theta_n)$ $(\Theta\in [0,2\pi)^n$ if $\mathbb{F}=\mathbb{C}, $ and $\Theta\in \{0,\pi\}^n$ if $\mathbb{F}=\mathbb{R})$ there exists $1\leq i\leq n,$ such that  $\alpha_{jl,ik}\geq 0$ for all $1\leq k\leq n, 1\leq j<l\leq n,$ where  $$\alpha_{jl,ik}=\cos(\theta_j-\theta_l+\xi_{ij}-\xi_{il})-\cos(\theta_j-\theta_l+\xi_{kj}-\xi_{kl}).$$
	\end{enumerate}
\end{theorem}
\begin{proof}
	$(1)\Rightarrow (2).$ From Lemma \ref{lem-inv} it follows that 
	$T=(k_je^{i\xi_{rj}})_{r,j=1}^n \text{ for some } k_j>0.$ 
	For $D=diag(\frac{1}{k_1},\ldots,\frac{1}{k_n}),$ $TD=(e^{i\xi_{rj}})_{r,j=1}^n.$ By Theorem \ref{th-li}, $D:\ell_1^n\to \ell_1^n$ preserves parallel pairs, and so does $TD:\ell_1^n\to \ell_\infty^n.$ 
	Now suppose $\Theta=(\theta_1,\theta_2,\ldots,\theta_n).$ 
	Since for each $$x\in \Big\{(|x_1|e^{i\theta_1},\ldots, |x_n|e^{i\theta_n}):|x_k|\geq 0,1\leq k\leq n\Big\},$$ $TDx$ attains its norm at $r$-th coordinate for some $1\leq r\leq n,$ using infinite Ramsey theorem \cite{R}, it follows that
	\begin{enumerate}[(a)]
		\item there  is a subset $A_s\subseteq (0,\infty)$ for each $1\leq s\leq n$ such that $0$ is a limit point of $A_s,$
		\item there exists $i\in \{1,2,\ldots,n\}$ such that $TDx$ attains its norm at $(TDx)_i,$ the $i$-th coordinate of $TDx,$  for all $x\in A:= \{(x_1e^{i\theta_1},\ldots, x_ne^{i\theta_n}):x_s\in A_s,1\leq s\leq n\}.$ 
	\end{enumerate}
Indeed, suppose $U=\{\frac{1}{k}:k\in \mathbb{N}\}.$  Define a map as follows. 
\[f:U^n\to \{1,2,\ldots,n\} \text{ by } f(b_1,\ldots,b_n)=j,\]  where $ j$ is the smallest coordinate at which $TD(b_1e^{i\theta_1},\ldots,b_ne^{i\theta_n})$  
attains its norm. 
Then by the infinite Ramsey theorem, there exist an infinite subset $V\subseteq U$ and $1\leq i\leq n,$ such that $f(a_1,\ldots,a_n)=i$ for all $a_1,\ldots,a_n\in V.$ Suppose $V=\{\frac{1}{r_k}:k\in \mathbb{N}\},$ and consider $A_s=V$ for all $1\leq s\leq n.$ Thus $(a)$ and $(b)$ hold.\\
	From  $(b)$ it follows that for all $1\leq k\leq n$ and $x\in A,$
	\begin{eqnarray*}
		&&|(TDx)_i|\geq|(TDx)_k|	\\
		&\Rightarrow &|\sum_{r=1}^ne^{i(\xi_{ir}+\theta_r)}x_r|\geq   |\sum_{r=1}^ne^{i(\xi_{kr}+\theta_r)}x_r|\\
		&\Rightarrow &\sum_{1\leq j<l\leq n}x_jx_l \Big\{\cos(\theta_j-\theta_l+\xi_{ij}-\xi_{il})-\cos(\theta_j-\theta_l+\xi_{kj}-\xi_{kl})\Big\}\geq 0\\
		&\Rightarrow &\sum_{1\leq j<l\leq n}x_jx_l \alpha_{jl,ik}\geq 0\\
		&\Rightarrow &x_jx_l \alpha_{jl,ik}\geq 0, \text{~taking } x_s\to 0 \text{ for all } s\in \{1,\ldots,n\}\setminus\{j,l\}\\
		&\Rightarrow & \alpha_{jl,ik}\geq 0,\text{ for all } 1\leq j<l\leq n. 
	\end{eqnarray*}
	$(2)\Rightarrow (1).$ Suppose $(x,y)$ is a parallel pair in $\ell_1^n.$ Then $x=(|x_1|e^{i\theta_1},\ldots, |x_n|e^{i\theta_n}),$ and $\lambda y=(|y_1|e^{i\theta_1},\ldots, |y_n|e^{i\theta_n})$ for some $\lambda\in \mathbb{T},$ and $\Theta=(\theta_1,\theta_2,\ldots,\theta_n).$ Now, 
	\begin{eqnarray*}
		&&	\alpha_{jl,ik}\geq 0,\text{ for all } 1\leq j<l\leq n\\
		&\Rightarrow &\sum_{1\leq j<l\leq n}|x_jx_l |\alpha_{jl,ik}\geq 0\\
		&\Rightarrow &|\sum_{r=1}^ne^{i(\xi_{ir}+\theta_r)}|x_r||\geq   |\sum_{r=1}^ne^{i(\xi_{kr}+\theta_r)}|x_r||\\
		&\Rightarrow &|(TDx)_i|\geq|(TDx)_k|, \text{  for all }1\leq k\leq n.
	\end{eqnarray*}
	Similarly, $|(TDy)_i|\geq|(TDy)_k|,$ for all $1\leq k\leq n.$ Thus $TDx$ and $TDy$ attain their norm at the $i$-th coordinate, and so $(TDx,TDy)$ is a parallel pair in $\ell_\infty^n.$ Therefore $TD:\ell_1^n\to \ell_\infty^n$ preserves parallel pairs. Moreover, since $D^{-1}:\ell_1^n\to \ell_1^n$ preserves parallel pairs, $T:\ell_1^n\to \ell_\infty^n$ preserves parallel pairs.
\end{proof}

As an immediate consequence of Theorem \ref{th-inv}, we  get the following characterization of parallel pair preservers in $\mathcal{L}(\ell_1^2,\ell_\infty^2).$

\begin{cor}
	Suppose $T=\begin{bmatrix}
		a&b\\
		c&d
	\end{bmatrix}.$  Then $T:\ell_1^2\to \ell_\infty^2$ preserves parallel pairs if and only if either of the following holds.
\begin{enumerate}
	\item $rank(T)\leq 1.$
	\item $T$ is invertible, and $|a|=|c|,$ $ |b|=|d|.$
	\end{enumerate}
\end{cor}
\begin{proof}
	Suppose $T$ preserves parallel pairs and $T$ is invertible. Then by Theorem \ref{th-inv}, there exists an invertible diagonal matrix say $D=diag(r,s)$ such that $TD=\begin{bmatrix}
		e^{i\xi_{11}}&e^{i\xi_{12}}\\
		e^{i\xi_{21}}&e^{i\xi_{22}}
	\end{bmatrix},$ that is, 	$T=\begin{bmatrix}
		\frac{1}{r}e^{i\xi_{11}}& \frac{1}{s} e^{i\xi_{12}}\\
		\frac{1}{r} e^{i\xi_{21}}& \frac{1}{s} e^{i\xi_{22}}
	\end{bmatrix}.$ Hence $|a|=|c|=\frac{1}{r},$ and $ |b|=|d|=\frac{1}{s}.$\\
	Conversely, if $(1)$ holds, then clearly $T$ preserves parallel pairs. Suppose $(2)$ holds. Assume that $|a|=|c|=\frac{1}{r},$ and $ |b|=|d|=\frac{1}{s}.$ Suppose $T=\begin{bmatrix}
		\frac{1}{r}e^{i\xi_{11}}& \frac{1}{s} e^{i\xi_{12}}\\
		\frac{1}{r} e^{i\xi_{21}}& \frac{1}{s} e^{i\xi_{22}}
	\end{bmatrix}.$ Then  $TD=\begin{bmatrix}
		e^{i\xi_{11}}&e^{i\xi_{12}}\\
		e^{i\xi_{21}}&e^{i\xi_{22}}
	\end{bmatrix},$ where $D=diag(r,s).$ Let $\Theta=(\theta_1,\theta_2).$ Note that $\alpha_{12,12}=-\alpha_{12,21}.$ Therefore for $\Theta,$ either $\alpha_{12,12}\geq 0$ or $\alpha_{12,21}\geq 0.$ Consequently, it follows from Theorem \ref{th-inv} that $T$ preserves parallel pairs.
\end{proof}

	Let $T:\ell_1^n\to \ell_\infty^m$ be a linear map, and $T=(t_{ij}).$ 	Suppose $ t_{ij}=|t_{ij}|e^{i\phi_{ij}},$ for all $1\leq i\leq m, 1\leq j\leq n.$  In what follows, we  always assume that $\|Te_j\|=k_j,$ for $1\leq j\leq n,$ and  $$C_{j1}=\{i\in\{1,\ldots,m\}:|t_{ij}|=k_j\},\quad C_{j2}=\{i\in\{1,\ldots,m\}:|t_{ij}|<k_j\}.$$
	Let $C=\{1,\ldots,m\}\setminus \Big(\cap_{j=1}^n C_{j1}\Big).$
 In the next proposition,  we first obtain a sufficient condition for  $T$ to preserve parallel pairs.
\begin{prop}\label{prop-nsuff}
Suppose for each $\Theta=(\theta_1,\ldots,\theta_n),$ $(\Theta\in [0,2\pi)^n$ if $\mathbb{F}=\mathbb{C},$ and $ \Theta\in \{0,\pi\}^n,$ if $\mathbb{F}=\mathbb{R}),$ there exists $s\in \cap _{1\leq k\leq n}C_{k1}$ such that for all $1\leq j<l\leq n,$ 
\begin{eqnarray*}
	\cos(\theta_j+\phi_{sj}-\theta_l-\phi_{sl})
\geq \max\Big\{\frac{|t_{rj}t_{rl}|}{k_jk_l}\cos(\theta_j+\phi_{rj}-\theta_l-\phi_{rl}):1\leq r\leq m\Big\}.
 \end{eqnarray*}
	Then $T$ preserves parallel pairs.
\end{prop}
\begin{proof}
	Suppose $(x,y)$ is a parallel pair in $\ell_1^n.$ Then $x=(|x_1|e^{i\theta_1},\ldots, |x_n|e^{i\theta_n}),$ and $\lambda y=(|y_1|e^{i\theta_1},\ldots, |y_n|e^{i\theta_n})$ for some $\lambda\in \mathbb{T},$ and $\Theta=(\theta_1,\theta_2,\ldots,\theta_n).$  Now for all $1\leq r\leq m,$ and $1\leq j<l\leq m,$ 
	\begin{eqnarray*}
	&&	\cos(\theta_j+\phi_{sj}-\theta_l-\phi_{sl})\geq \frac{|t_{rj}t_{rl}|}{k_jk_l}\cos(\theta_j+\phi_{rj}-\theta_l-\phi_{rl})\\
	&\Rightarrow& 	\sum_{j=1}^m |k_jx_j|^2+2\sum_{1\leq j<l\leq m}|k_jx_jk_lx_l|\cos(\theta_j+\phi_{sj}-\theta_l-\phi_{sl})\geq \\
	&&	\quad \sum_{j=1}^m |t_{rj}x_j|^2+2\sum_{1\leq j<l\leq m}|t_{rj}x_jt_{rl}x_l|\cos(\theta_j+\phi_{rj}-\theta_l-\phi_{rl})\\
	&\Rightarrow& 		\Big\{\sum_{j=1}^m |t_{sj}x_j|\cos(\theta_j+\phi_{sj})\Big\}^2+	\Big\{\sum_{j=1}^m |t_{sj}x_j|\sin(\theta_j+\phi_{sj})\Big\}^2\geq \\
	&&	\quad 	\Big\{\sum_{j=1}^m |t_{rj}x_j|\cos(\theta_j+\phi_{rj})\Big\}^2+	\Big\{\sum_{j=1}^m |t_{rj}x_j|\sin(\theta_j+\phi_{rj})\Big\}^2, ~(\text{since } k_j=|t_{sj}|)\\
			&\Rightarrow& 		|\sum_{j=1}^m |t_{sj}x_j|e^{i(\theta_j+\phi_{sj})}|^2\geq 	|\sum_{j=1}^m |t_{rj}x_j|e^{i(\theta_j+\phi_{rj})}|^2\\
		&\Rightarrow& 		|\sum_{j=1}^m t_{sj}x_j|\geq 	|\sum_{j=1}^m t_{rj}x_j| . 
	\end{eqnarray*}
	Similarly $	|\sum_{j=1}^m t_{sj}y_j|\geq 	|\sum_{j=1}^m t_{rj}y_j|. $ Therefore both $Tx,Ty$ attain their norm at the $s$-th coordinate, and so $(Tx,Ty)$ is a parallel pair in $\ell_\infty^m.$ Thus $T$ preserves parallel pairs.
\end{proof}
 The following sufficient condition for parallel pair preservers in $\mathcal{L}(\ell_1^2,\ell_\infty^m)$ immediately follows from Proposition \ref{prop-nsuff}.  For the sake of completeness, we include a proof here.   In what follows, we denote  $\alpha_s=\frac{|t_{s1}t_{s2}|}{k_1k_2}$ for all  $1\leq s\leq m.$
%
\begin{cor}\label{cor-03}
	Suppose $T:\ell_1^2\to \ell_\infty^m$ is given by $T=(t_{ij}),$ where $t_{i1}\geq 0,$ and $t_{i2}=|t_{i2}|e^{i\phi_i}$ for $1\leq i\leq m.$ 	Assume either of the following holds:
	\begin{enumerate}
		\item $|C_{11}\cap C_{21}|=m.$ 
		\item $1<|C_{11}\cap C_{21}|<m.$ Moreover, if $\mathbb{F}=\mathbb{C},$ then for each  $\theta\in [0,2\pi)$
		\begin{equation}\label{eq-gen}
			\max\{\cos(\theta+\phi_j):j\in C_{11}\cap C_{21}\} \geq \max\Big\{\alpha_r\cos(\theta+\phi_r):r\in C \Big\}.	
		\end{equation}	
	If $\mathbb{F}=\mathbb{R},$ then  there exist $i,j\in C_{11}\cap C_{21}$ such that $\phi_i=0,\phi_j=\pi.$
	\end{enumerate}
%
Then $T$ preserves parallel pairs.
\end{cor}
\begin{proof}
$(1)$ Since $C_{11}\cap C_{21}=\{1,\ldots,m\},$ $|t_{r1}|=k_1,$ and $|t_{r2}|=k_2$ for all $1\leq r\leq m.$ Now the proof immediately follows from Proposition \ref{prop-nsuff} by observing that for each $\Theta=(\theta_1,\theta_2),$ there exists $1\leq s\leq m$ such that $$\max\{\cos(\theta_1-\theta_2-\phi_{r}):1\leq r\leq m\}=\cos(\theta_1-\theta_2-\phi_{s}).$$ 	

$(2)$ First consider $\mathbb{F}=\mathbb{C}.$  Suppose $\Theta=(\theta_1,\theta_2)\in[0,2\pi)^2.$ 
Let $\theta=\theta_2-\theta_1,$ and  $$\max\{\cos(\theta+\phi_j):j\in C_{11}\cap C_{21}\}=\cos(\theta+\phi_s)$$ for some $s\in C_{11}\cap C_{21}.$ By \eqref{eq-gen}, and observing that $\alpha_r=1$ for $r\in C_{11}\cap C_{21},$ we have
\begin{eqnarray*}
\cos(\theta+\phi_s)&\geq& \max\Big\{\alpha_r\cos(\theta+\phi_r):1\leq r\leq m \Big\}\\
\Rightarrow \cos(\theta_1-\theta_2-\phi_s)&\geq& \max\Big\{\alpha_r\cos(\theta_1-\theta_2-\phi_r):1\leq r\leq m \Big\}.
\end{eqnarray*}
Therefore by Proposition \ref{prop-nsuff}, $T$ preserves parallel pairs.\\
Now consider $\mathbb{F}=\mathbb{R}.$ For $\Theta=(\theta_1,\theta_2)\in\{0,\pi\}^2,$ $\theta_1-\theta_2\in \{0,\pi,-\pi\}.$ Since $$\max\{\cos(\theta_1-\theta_2-\phi_i), \cos(\theta_1-\theta_2-\phi_j)\}=1\geq \max\Big\{\alpha_r\cos(\theta_1-\theta_2-\phi_r):1\leq r\leq m \Big\},$$
Proposition \ref{prop-nsuff} implies that $T$ preserves parallel pairs.
\end{proof}

Note that, if $T:\ell_1^n\to \ell_\infty^m$ preserves parallel pairs, then so does the restriction of $T$ on $\ell_1^2, $ say $S.$  Indeed $S=TR,$ where $R: \ell_1^2\to \ell_1^n$ is defined as $Re_1=e_1,Re_2=e_2. $ By Theorem \ref{th-li}, $R$ preserves parallel pairs, and so does $TR:\ell_1^2\to \ell_\infty^m.$  Therefore to get necessary conditions for parallel pair preserving maps in $\mathcal{L}(\ell_1^n,\ell_\infty^m)$, we only consider $T:\ell_1^2\to \ell_\infty^m,$ where $m\geq 3.$  
Observe that if $P_1,P_2$ are generalized permutation matrices of size $m$ and $2$ respectively, then by \cite{LTWW}, $P_1:\ell_\infty^m\to \ell_\infty^m$ and $P_2:\ell_1^2\to \ell_1^2$ preserve parallel pairs. Therefore $T$ preserves parallel pairs if and only if $P_1TP_2:\ell_1^2\to \ell_\infty^m$ does. \\
Assume $T=\begin{bmatrix}
	t_{11}&t_{12}\\
	t_{21}&t_{22}\\
	\vdots &\vdots\\
	t_{m1}&t_{m2}
\end{bmatrix}:\ell_1^2\to \ell_\infty^m$ preserves parallel pairs. If necessary, by pre and post multiplying $T$ with generalized permutation matrices of appropriate size, without loss of generality, we may assume that $t_{j1}\geq 0,$ and $ t_{j2}=|t_{j2}|e^{i\phi_j},$ for all $1\leq j\leq m.$ Observe that  $C_{11}\cap C_{21}\neq \emptyset,$ for otherwise, $(Te_1,Te_2)$ is not a parallel pair in $\ell_\infty^m,$ which contradicts that  $(e_1,e_2)$ is a parallel pair in $\ell_1^2.$  Assume that $|C_{11}\cap C_{21}|<m,$ and  $C=\{1,\ldots,m\}\setminus(C_{11}\cap C_{21}).$ 
 Express $C$ as the union of $3$ disjoint sets, namely,   $$C= \Big(C_{12}\cap C_{21}\Big) \cup \Big(C_{11}\cap C_{22}\Big)\cup \Big(C_{12}\cap C_{22}\Big).$$ 	For $s\in C,$
 consider the sets 
\begin{eqnarray*}
A_s^+&:=&\{i\in  C_{11}\cap C_{21}:\cos(\phi_i)-\alpha_s \cos(\phi_s)>0\},\\
A_s^-&:=&\{i\in  C_{11}\cap C_{21}:\cos(\phi_i)-\alpha_s  \cos(\phi_s)<0\},\\
A_s^0&:=&\{i\in  C_{11}\cap C_{21}:\cos(\phi_i)=\alpha_s \cos(\phi_s)\}.
\end{eqnarray*} 
For necessary condition on $T,$ we consider the cases 
\[(C_{11}\cap C_{22})\cup (C_{12}\cap C_{21})\neq \emptyset, \quad \text{and } \quad C_{12}\cap C_{22}\neq \emptyset \] separately. 
\begin{prop}\label{prop-nec}
	Suppose $(C_{11}\cap C_{22})\cup (C_{12}\cap C_{21})\neq \emptyset,$ and $T$ preserves parallel pairs. 
	Then 
\begin{enumerate}
\item For each $\theta,$ $(\theta\in [0,2\pi)$ if $\mathbb{F}=\mathbb{C},$ and $\theta\in \{0,\pi\}$ if $\mathbb{F}=\mathbb{R})$ 
\begin{equation}\label{eq-06}
	\begin{split}
\max\{\cos(\theta+\phi_j):j\in C_{11}\cap C_{21}\}
\geq  \max\Big\{\alpha_s\cos(\theta+\phi_s):s\in C_{11}\cup C_{21}\Big\}.	
\end{split}
\end{equation}	
\item $|C_{11}\cap C_{21}|\geq 2.$
\end{enumerate}
\end{prop}
\begin{proof}
	
$(1)$  We first show that if $C_{12}\cap C_{21}\neq \emptyset,$ then $$\max\{\cos(\theta+\phi_j):j\in C_{11}\cap C_{21} \}
\geq \max\Big\{\alpha_i\cos(\theta+\phi_i): i\in C_{12}\cap C_{21}\Big\}.$$ 
Assume to the contrary that there exists $i\in C_{12}\cap C_{21}$ such that 
\[\max\{\cos(\theta+\phi_j):j\in C_{11}\cap C_{21} \}<\frac{|t_{i1}|}{k_1}\cos(\theta+\phi_i).\]	
Choose $n_0\in \mathbb{N}$ such that for all $j\in C_{11}\cap C_{21},$
\[\frac{2k_2n_0}{k_1}\Big\{\cos(\theta+\phi_j)-\frac{|t_{i1}|}{k_1}\cos(\theta+\phi_i)\Big\}<\frac{|t_{i1}|^2}{k_1^2}-1.\]
Consequently, for  $n_1\geq n_0,$
\begin{equation}\label{eq-01}
	\begin{aligned}
k_1^2+2k_1k_2n_1\cos(\theta+\phi_j)+k_2^2n_1^2&<|t_{i1}|^2+2	k_2n_1|t_{i1}|\cos(\theta+\phi_i)+k_2^2n_1^2\\
\Rightarrow|k_1+n_1k_2e^{i(\theta+\phi_j)}|&<||t_{i1}|+n_1k_2e^{i(\theta+\phi_i)}|\\
\Rightarrow|t_{j1}+n_1e^{i\theta}t_{j2}|&<|t_{i1}+n_1e^{i\theta}t_{i2}|.
\end{aligned}
\end{equation}
Choose 
$n_2\in \mathbb{N}$ sufficiently large so that $n_2>n_0,$ and for all $j\in C_{11}\cap C_{22},$
\begin{eqnarray*}
k_1^2-|t_{i1}|^2+n_2^2(|t_{j2}|^2-k_2^2)&<& 2n_2\Big(k_2	|t_{i1}|\cos(\theta+\phi_i)-k_1|t_{j2}|\cos(\theta+\phi_j)\Big).
\end{eqnarray*}
Note that such a choice of $n_2$ is possible, since $(|t_{j2}|^2-k_2^2)<0,$ for all $j\in C_{11}\cap C_{22}.$ Therefore, 
\begin{equation}\label{eq-02}
		\begin{split}
	k_1^2+2n_2k_1|t_{j2}|\cos(\theta+\phi_j)+n_2^2|t_{j2}|^2&< |t_{i1}|^2+2n_2k_2	|t_{i1}|\cos(\theta+\phi_i)+n_2^2k_2^2\\
	\Rightarrow |k_1+n_2|t_{j2}|e^{i(\theta+\phi_j)}|&<||t_{i1}|+n_2k_2e^{i(\theta+\phi_i)}|\\
	\Rightarrow|t_{j1}+n_2e^{i\theta}t_{j2}|&<|t_{i1}+n_2e^{i\theta}t_{i2}|.
	\end{split}
\end{equation}
From \eqref{eq-01} and \eqref{eq-02}, we get, for all $j\in C_{11},$
\[|t_{j1}+n_2e^{i\theta}t_{j2}|<|t_{i1}+n_2e^{i\theta}t_{i2}|.\]
Therefore, $T(1,n_2e^{i\theta})$ does not attain its norm at the $j$-th coordinate, where $j\in C_{11}.$  So $\Big(T(1,n_2e^{i\theta}),T(1,0)\Big)$ is not a parallel pair in $\ell_\infty^m,$ whereas, $\Big((1,n_2e^{i\theta}),(1,0)\Big)$ is a parallel pair in $\ell_1^2.$ This contradiction proves that 
\[\max\{\cos(\theta+\phi_j):j\in C_{11}\cap C_{21}\}\\
\geq \max\Big\{\alpha_i\cos(\theta+\phi_i):i\in C_{12}\cap C_{21}\Big\}.\]
Next we show that  if $C_{11}\cap C_{22}\neq \emptyset,$ then $$\max\{\cos(\theta+\phi_j):j\in C_{11}\cap C_{21}\}
\geq \max\Big\{\alpha_l\cos(\theta+\phi_l): l\in C_{11}\cap C_{22}\Big\}.$$ Assume to the contrary that there exists $l\in C_{11}\cap C_{22}$ such that 
\[\max\{\cos(\theta+\phi_j):j\in C_{11}\cap C_{21}\}<\frac{|t_{l2}|}{k_2}\cos(\theta+\phi_l).\]	
Choose $m_0\in \mathbb{N}$ such that for all $j\in C_{11}\cap C_{21},$
\[\frac{2k_1m_0}{k_2}\Big\{\cos(\theta+\phi_j)-\frac{|t_{l2}|}{k_2}\cos(\theta+\phi_l)\Big\}<\frac{|t_{l2}|^2}{k_2^2}-1,\]
and for all $j\in C_{12}\cap C_{21},$
\begin{eqnarray*}
	m_0^2(|t_{j1}|^2-k_1^2)+(k_2^2-|t_{l2}|^2)+2m_0\Big(k_2	|t_{j1}|\cos(\theta+\phi_j)-k_1|t_{l2}|\cos(\theta+\phi_l)\Big)<0.
\end{eqnarray*}
Proceeding as previous, it is easy to check that  for all $j\in C_{21},$
\[|m_0t_{j1}+e^{i\theta}t_{j2}|<|m_0t_{l1}+e^{i\theta}t_{l2}|.\]
Therefore, $T(m_0,e^{i\theta})$ does not attain its norm at the $j$-th coordinate, where $j\in C_{21},$ and so $\Big(T(m_0,e^{i\theta}),T(0,1)\Big)$ is not a parallel pair in $\ell_\infty^m,$ whereas, $\Big((m_0,e^{i\theta}),(0,1)\Big)$ is a parallel pair in $\ell_1^2.$ This contradiction completes the proof.

$(2)$ 
Suppose for a contradiction that  $C_{11}\cap C_{21}= \{i\}.$ \\
For $\mathbb{F}=\mathbb{R},$  choose $\theta\in \{0,\pi\}$ such that $\cos(\theta+\phi_i)=-1.$ For this choice of $\theta,$ \eqref{eq-06} is not satisfied, proving that   $|C_{11}\cap C_{21}|\geq 2.$\\
Suppose $\mathbb{F}=\mathbb{C}.$ Let $s\in(C_{11}\cap C_{22})\cup (C_{12}\cap C_{21}).$ Then at least two of the sets $A_s^+,A_s^-,A_s^0$ are empty. First assume $A_s^+=A_s^-=\emptyset,$ that is, $A_s^0=\{i\}.$ Clearly $\sin(\phi_i)\neq \alpha_s\sin(\phi_s),$ for otherwise, $\alpha_s=1$ contradicting that $s\in C.$ 
Choose $\theta\in(0,2\pi)$ such that $$\sin(\theta)\Big(\sin(\phi_i)-\alpha_s\sin(\phi_s)\Big)>0 .$$ Then  
\begin{eqnarray*}
	\cos(\theta)\Big(\cos(\phi_i)-\alpha_s \cos(\phi_s)\Big)=0&<&\sin(\theta) \Big(\sin(\phi_i)-\alpha_s \sin(\phi_s)\Big)\\
	\Rightarrow \cos(\theta-\phi_i)&<&\alpha_s\cos(\theta-\phi_s), 
\end{eqnarray*} 
which contradicts \eqref{eq-06}.  Now assume $A_s^+=A_s^0=\emptyset.$ Then $A_s^-=\{i\}.$ Choose $\theta\in (0,\pi)$ such that 
\begin{eqnarray*}
	\cos(\theta)&>& \sin(\theta) \frac{\sin(\phi_i)-\alpha_s \sin(\phi_s)}{\cos(\phi_i)-\alpha_s \cos(\phi_s)}\\
	\Rightarrow \cos(\theta-\phi_i)&<&\alpha_s\cos(\theta-\phi_s),
\end{eqnarray*}
which again contradicts \eqref{eq-06}. Similarly $A_s^-=A_s^0=\emptyset$ produces a contradiction. Therefore, at most one of the sets $A_s^+,A_s^-,A_s^0$ is empty, and so $|C_{11}\cap C_{21}|\geq 2.$\\

\end{proof}

Combining Proposition \ref{prop-nec} and Corollary \ref{cor-03}, we immediately get the following characterization.
\begin{cor}\label{cor-com}
	Suppose $C\neq \emptyset,$ and $C_{12}\cap C_{22}=\emptyset.$
	Then  $T$ preserves parallel pairs if and only if the following hold.
	\begin{enumerate}
		\item If $\mathbb{F}=\mathbb{C},$ then for each $\theta\in [0,2\pi),$
		\begin{equation}\label{eq-006}
			\begin{split}
				\max\{\cos(\theta+\phi_j):j\in C_{11}\cap C_{21}\}
				\geq  \max\Big\{\alpha_s\cos(\theta+\phi_s):s\in C\Big\}.	
			\end{split}
		\end{equation}	
	If $\mathbb{F}=\mathbb{R},$ then  there exist $i,j\in C_{11}\cap C_{21}$ such that $\phi_i=0,\phi_j=\pi.$
		\item $|C_{11}\cap C_{21}|\geq 2.$
	\end{enumerate}
\end{cor}

Now we consider the case  $C_{12}\cap C_{22}\neq \emptyset.$ 
For $\theta,$ where $\theta\in [0,2\pi)$ if $\mathbb{F}=\mathbb{C},$ and $\theta\in \{0,\pi\}$ if $\mathbb{F}=\mathbb{R},$ define 
\begin{eqnarray*}
	A_\theta&:=&\Big\{s\in  C_{11}\cap C_{21}:\cos(\theta+\phi_s)=\max_{j\in C_{11}\cap C_{21}}\cos(\theta+\phi_j)\Big\}\\
	B_{\theta}&:=&\Big\{r\in  C_{12}\cap C_{22}:\cos(\theta+\phi_j)<\alpha_r\cos(\theta+\phi_r), j\in A_\theta \Big\}
\end{eqnarray*}
For $r\in B_\theta,$ define
\begin{eqnarray*}
	&&\beta_{r}=k_1k_2\cos(\theta+\phi_j)-|t_{r1}t_{r2}|\cos(\theta+\phi_r), \\
	&& \zeta_{r}=|t_{r1}|^2-k_1^2,\quad \eta_r=|t_{r2}|^2-k_2^2,\\
	&&\lambda_{r}=\frac{\beta_{r}+\sqrt{\beta_{r}^2-\zeta_r\eta_r}}{\zeta_r},\quad \mu_{r}=\frac{\beta_{r}-\sqrt{\beta_{r}^2-\zeta_r\eta_r}}{\zeta_r}.
\end{eqnarray*}
Observe that, if $\mathbb{F}=\mathbb{R},$ and $B_\theta\neq \emptyset,$ then  $	\max_{j\in C_{11}\cap C_{21}}\cos(\theta+\phi_j)=-1,$ and so  $A_\theta=C_{11}\cap C_{21},$ and $B_\theta=C_{12}\cap C_{22}.$

\begin{prop}\label{prop-04}
	Suppose $T$ preserves parallel pairs. Let $C_{12}\cap C_{22}\neq \emptyset.$ Assume $B_\theta\neq \emptyset,$ where $\theta\in [0,2\pi)$ if $\mathbb{F}=\mathbb{C},$ and $\theta\in \{0,\pi\}$ if $\mathbb{F}=\mathbb{R}.$   Then for all $r\in B_\theta,$ $\beta_{r}^2-\zeta_r\eta_r\leq0.$
\end{prop}
\begin{proof}
	Suppose for a contradiction, there exists $r\in B_\theta$ such that  $\beta_{r}^2-\zeta_r\eta_r>0.$ From $\zeta_r<0,$ and $\beta_{r}<0,$ it follows that $0<\lambda_{r}<\mu_{r}.$ Choose $x\in(\lambda_{r},\mu_{r}).$ Then for all $j\in A_\theta,$ 	
	\begin{eqnarray*}
		(x-\lambda_{r})(x-\mu_{r})& <	&0\\
		\Rightarrow  \zeta_rx^2-2x\beta_{r}+\eta_r&>&0\\
		\Rightarrow  |t_{r1}|^2x^2+2x|t_{r1}t_{r2}|\cos(\theta+\phi_r)+|t_{r2}|^2&>&k_1^2x^2+2xk_1k_2\cos(\theta+\phi_j)+k_2^2\\
		\Rightarrow |t_{r1}x+|t_{r2}|e^{i(\theta+\phi_r)}|&>&|k_1x+k_2e^{i(\theta+\phi_j)}|\\
		\Rightarrow |t_{r1}x+t_{r2}e^{i\theta}|&>&|t_{j1}x+t_{j2}e^{i\theta}|.
	\end{eqnarray*}
	On the other hand, from Proposition \ref{prop-nec}, it follows that for all $j\in A_{\theta},$ and for all $i\in C_{11}\cup C_{21},$
		\begin{eqnarray*}
		\cos(\theta+\phi_j)&\geq& \alpha_i\cos(\theta+\phi_s)\\
		\Rightarrow k_1^2x^2+2k_1k_2x	\cos(\theta+\phi_j)+k_2^2&\geq& |t_{i1}x|^2+2|t_{i1}t_{i2}|x\cos(\theta+\phi_i)+|t_{i2}|^2\\
		\Rightarrow |k_1x+k_2e^{i(\theta+\phi_j)}|&\geq& |t_{i1}x+|t_{i2}|e^{i(\theta+\phi_i)}|\\
		\Rightarrow |t_{j1}x+t_{j2}e^{i\theta}|&\geq &|t_{i1}x+t_{i2}e^{i\theta}|.
	\end{eqnarray*}
	Consequently,
	\[|t_{r1}x+t_{r2}e^{i\theta}|> |t_{i1}x+t_{i2}e^{i\theta}|.\]
	Therefore, $T(x,e^{i\theta})$ does not not attain its norm at the $i$-th coordinate, where $i\in C_{11}\cup C_{21}.$ Since $T(1,0)$ attains its norm only on $C_{11},$ $\Big(T(1,0), T(x,e^{i\theta})\Big)$ is not a parallel pair in $\ell_\infty^m,$ whereas $\Big((1,0), (x,e^{i\theta})\Big)$ is a parallel pair in $\ell_1^2.$ This contradiction proves the result.
\end{proof}
In the next two theorems, we characterize the parallel pair preservers of $\mathcal{L}(\ell_1^2,\ell_\infty^m)$ in the most general case. We first consider the complex scalars.
\begin{theorem}\label{th-com}
Suppose $\mathbb{F}=\mathbb{C}.$ Then $T$ preserves parallel pairs if and only if the following holds.
	\begin{enumerate}
		\item If $(C_{11}\cap C_{22})\cup (C_{12}\cap C_{21})\neq \emptyset,$  then $|C_{11}\cap C_{21}|\geq 2,$ and  for each $\theta\in [0,2\pi), $ \eqref{eq-06} holds.
		\item If $C_{12}\cap C_{22}\neq \emptyset,$  then for each $\theta\in [0,2\pi),$ either $B_\theta=\emptyset$ or   $\beta_{r}^2-\zeta_r\eta_r\leq0$ for all $r\in B_{\theta}.$
	\end{enumerate}
\end{theorem}
\begin{proof}
	The necessary part follows from  Propositions \ref{prop-nec}  and  \ref{prop-04}. 	For the sufficient part, suppose $\Big((\alpha,\beta),(\gamma,\delta)\Big)$ is a parallel pair in $\ell_1^2.$ If $\alpha\beta=\gamma\delta=0,$ then $(\alpha,\beta),(\gamma,\delta)\in \{\mu e_1,\mu e_2:\mu\in \mathbb{T}\},$ and so  $T(\alpha,\beta),T(\gamma,\delta)\in  \{\mu Te_1,\mu Te_2:\mu\in \mathbb{T}\}.$ Since both $Te_1,Te_2$ attain their norm on $C_{11}\cap C_{21}$, $\Big(Te_1,Te_2\Big)$ is a parallel pair in $\ell_\infty^m,$ and so is $\Big(T(\alpha,\beta),T(\gamma,\delta)\Big).$ Suppose $\alpha\beta\neq 0.$ Since parallel pairs satisfy homogeneity, without loss of generality, we may assume that $\alpha> 0.$ Suppose $\beta=|\beta|e^{i\theta}.$ 
	So there exists $\lambda\in \mathbb{T}$ such that 
	 $\lambda\gamma=|\gamma|,$ and $\lambda e^{-i\theta}\delta=|\delta|.$ \\ 
	  
	 If $B_\theta=\emptyset,$ then \eqref{eq-gen} holds. For all $s\in C,$ and $j\in A_\theta,$ from $$	\cos(\theta+\phi_j)\geq \alpha_s\cos(\theta+\phi_s),$$ we get
	 \begin{equation}\label{eq-n0} 
	 |t_{j1}\alpha+t_{j2}\beta|\geq |t_{s1}\alpha+t_{s2}\beta| \quad \text{and } \quad 	|t_{j1}\gamma+t_{j2}\delta|\geq |t_{s1}\gamma+t_{s2}\delta |.
	 \end{equation}
	 Therefore both $T(\alpha,\beta),$ and $ T(\gamma,\delta)$ attain their norm at the $j$-th coordinate. So $\Big(T(\alpha,\beta),T(\gamma,\delta)\Big)$ is a parallel pair in $\ell_\infty^m.$ \\
	 
	 Suppose $B_\theta\neq \emptyset.$ Since for  $r\in (C_{12}\cap C_{22})\setminus B_\theta,$ 
	  $$\cos(\theta+\phi_j)\geq\alpha_r\cos(\theta+\phi_r),$$ using  \eqref{eq-06}, we have
	\[\max\{\cos(\theta+\phi_i):i\in C_{11}\cap C_{21}\}\geq \max\{\alpha_s\cos(\theta+\phi_s):s\in C_{11}\cup C_{21}\cup B_\theta^c\}.\] 
	For all $s\in C_{11}\cup C_{21}\cup B_\theta^c,$ and $j\in A_\theta,$  $$	\cos(\theta+\phi_j)\geq \alpha_s\cos(\theta+\phi_s),$$ implies \eqref{eq-n0}.
	On the other hand, if $r\in B_{\theta},$ then $\beta_{r}^2-\zeta_r\eta_r\leq0.$ So for all $x>0,$ and $j\in A_\theta$
	\begin{eqnarray*}
		\zeta_rx^2-2x\beta_{r}+\eta_r&\leq &0\\
		\Rightarrow |t_{r1}x+t_{r2}e^{i\theta}|&\leq&|t_{j1}x+t_{j2}e^{i\theta}|.
	\end{eqnarray*}
	In particular, considering $x=\frac{\alpha}{|\beta|},$ we have 
	\[ |t_{r1}\alpha+t_{r2}\beta|\leq|t_{j1}\alpha+t_{j2}\beta|.\]
	Thus $T(\alpha,\beta)$ attains its norm at the $j$-th coordinate. If $\gamma\delta=0,$ then clearly $T(\gamma,\delta)$ attains its norm at the $j$-th coordinate. If $\gamma\delta\neq 0,$ then considering $x=\frac{|\gamma|}{|\delta|},$ we have  
	\begin{eqnarray*}
		|t_{r1}|\gamma|+t_{r2}|\delta|e^{i\theta}|&\leq&|t_{j1}|\gamma|+t_{j2}|\delta|e^{i\theta}|\\
		\Rightarrow |t_{r1}\gamma+t_{r2}\delta|&\leq&|t_{j1}\gamma+t_{j2}\delta|,
	\end{eqnarray*}
	which implies that $T(\gamma,\delta)$ attains its norm at the $j$-th coordinate. 
	Thus in each case, $\Big(T(\alpha,\beta),T(\gamma,\delta)\Big)$ is a parallel pair in $\ell_\infty^m.$ 
\end{proof}
The following theorem characterizes the parallel pair preservers of $\mathcal{L}(\ell_1^2,\ell_\infty^m)$ for real scalars.
\begin{theorem}\label{th-comr}
	Suppose $\mathbb{F}=\mathbb{R}.$ Then $T$ preserves parallel pairs if and only if either $\phi_i=0,\phi_j=\pi$ for some $i,j\in C_{11}\cap C_{21}$ or the following holds.
	\begin{enumerate}
		\item $\phi_i=\phi_j$ for all  $i,j\in C_{11}\cap C_{21}.$
		\item  $(C_{11}\cap C_{22})\cup (C_{12}\cap C_{21})= \emptyset.$  
		\item For $r\in C_{12}\cap C_{22},$  either $t_{r1}=t_{r2}=0$ or $|t_{r2}|k_1=t_{r1}k_2$ and $\phi_r=\phi_i,$ $i\in  C_{11}\cap C_{21}.$
	\end{enumerate}
\end{theorem}
\begin{proof}
\textbf{Necessary part:} Suppose $(1)$ holds. 
Choose $\theta\in \{0,\pi\}$ such that $$\max\{\cos(\theta+\phi_j):j\in C_{11}\cap C_{21}\}=-1.$$ Then $(2)$ follows from \eqref{eq-06}. Moreover,  $r\in C_{12}\cap C_{22}$ implies that $r\in B_\theta.$ By Proposition \ref{prop-04}, $\beta_r^2-\zeta_r\eta_r\leq 0,$ which is equivalent to $(3).$\\
	\textbf{Sufficient part:} If $\phi_i=0,\phi_j=\pi$ for some $i,j\in C_{11}\cap C_{21},$ then by Corollary \ref{cor-03}, $T$ preserves parallel pairs. On the other hand, assume $(1)-(3).$  Let $\theta\in \{0,\pi\}.$ If $\cos(\theta+\phi_i)=1,$ for all $i\in C_{11}\cap C_{21},$ then \eqref{eq-gen} holds, and so $T$ preserves parallel pairs. Otherwise for this $\theta,$ by $(3)$  $\beta_r^2-\zeta_r\eta_r\leq 0$ for all $r\in C_{12}\cap C_{22}.$ Now, proceeding similarly as  Theorem \ref{th-com}, we get that $T$ preserves parallel pairs.
\end{proof}

In Theorem \ref{th-comr}, the characterization of parallel pair preservers $T\in \mathcal{L}(\ell_1^2,\ell_\infty^m)$ for real scalars involves only the entries of the matrix $T.$ Whereas using Theorem \ref{th-com} or Corollary \ref{cor-03}, it is difficult to construct parallel pair preserving maps in $\mathcal{L}(\ell_1^2,\ell_\infty^m)$ for complex scalars.
Therefore, our next goal is to obtain an equivalent condition for \eqref{eq-gen} involving only the entries of the matrix of $T$ for complex scalars. For this purpose, we first prove the following lemma. 

\begin{lemma}\label{lem-gen}
	
	Let $\mathbb{F}=\mathbb{C}.$ Suppose   \eqref{eq-gen} holds. Then  for each  $s\in C,$ at least one  of the following holds.
	\begin{enumerate}
		\item The sets	$A_s^+$ and  $A_s^-$ are non-empty.
		\item   $|A_s^0|\geq 2,$ and there exist distinct $i,j\in  A_s^0$ such that $$\Big(\sin(\phi_i)-\alpha_s\sin(\phi_s)\Big)\Big(\sin(\phi_j)-\alpha_s\sin(\phi_s)\Big)<0.$$
	\end{enumerate}
\end{lemma}
\begin{proof}
	Since  \eqref{eq-gen} holds, by Corollary \ref{cor-03}, $T$ preserves parallel pairs. 
	Suppose for some $s\in  C,$ either $A_s^+=\emptyset,$ or $A_s^-=\emptyset.$	Without loss of generality, assume $A_s^+=\emptyset.$  We consider the cases $A_s^-=\emptyset,$ and $A_s^-\neq \emptyset$ separately. \\
	
	\textbf{Case a:} Let  $A_s^-=\emptyset.$ Then $C_{11}\cap C_{21}=A_s^0.$ 
	 If $j\in A_s^0,$ then $\sin(\phi_j)\neq \alpha_s\sin(\phi_s),$ for otherwise, $\alpha_s=1$ contradicting that $s\in C.$ Assume to the contrary that 
	for all distinct $i,j\in  A_s^0,$  $$\Big(\sin(\phi_i)-\alpha_s\sin(\phi_s)\Big)\Big(\sin(\phi_j)-\alpha_s\sin(\phi_s)\Big)>0,$$ that is, for all $i,j\in A_s^0,$
	\[sgn\Big(\sin(\phi_i)-\alpha_s\sin(\phi_s)\Big)=sgn\Big(\sin(\phi_j)-\alpha_s\sin(\phi_s)\Big).\]
	Choose $\theta\in(0,2\pi)$ such that $$\sin(\theta)\Big(\sin(\phi_i)-\alpha_s\sin(\phi_s)\Big)>0 ~\text{ for all } i\in  A_s^0.$$ Proceeding as the proof of $(2)$ of Proposition \ref{prop-nec}, we get 
	\[ \max\{ \cos(\theta-\phi_i):i\in A_s^0\}< \max \{\alpha_s\cos(\theta-\phi_s):s\in C\},\]
	which contradicts \eqref{eq-gen}. Therefore $(2)$ holds.\\
	
	\textbf{Case b:} Let  $A_s^-\neq\emptyset.$ We first show that $|A_s^0|\geq 2.$   Assume
	\begin{eqnarray*}
		\delta_1&=&\min\Big\{\frac{\sin(\phi_i)-\alpha_s \sin(\phi_s)}{\cos(\phi_i)-\alpha_s \cos(\phi_s)}:i\in A_s^-\Big\}, \\
		\delta_2&=&\max\Big\{\frac{\sin(\phi_i)-\alpha_s \sin(\phi_s)}{\cos(\phi_i)-\alpha_s \cos(\phi_s)}:i\in A_s^-\Big\}.
	\end{eqnarray*}
	Suppose for a contradiction that $|A_s^0|\leq 1.$ We consider the cases $A_s^0=\emptyset$ and  $|A_s^0|= 1$ separately. \\
	
	\textbf{Case b1:} Assume that $A_s^0=\emptyset,$ that is, $C_{11}\cap C_{21}=A_s^-.$ Choose $\theta\in (0,\pi)$ such that $\cos(\theta)>\delta_2 \sin(\theta).$ Then for all $i\in C_{11}\cap C_{21},$
	\begin{eqnarray*}
		\cos(\theta)&>&\delta_2 \sin(\theta)\\
		\Rightarrow \cos(\theta)&>& \sin(\theta) \frac{\sin(\phi_i)-\alpha_s \sin(\phi_s)}{\cos(\phi_i)-\alpha_s \cos(\phi_s)}\\
		\Rightarrow \cos(\theta-\phi_i)&<&\alpha_s\cos(\theta-\phi_s),
	\end{eqnarray*}
	which contradicts \eqref{eq-gen}.\\
	
	\textbf{Case b2:} Assume $A_s^0=\{j\},$ that is, $C_{11}\cap C_{21}=A_s^-\cup \{j\}.$ Observe that  $\sin(\phi_j)\neq \alpha_s\sin(\phi_s),$ for otherwise, $\alpha_s=1,$ contradicting that $s\in C.$\\
	
	If $\sin(\phi_j)- \alpha_s\sin(\phi_s)>0,$ choose $\theta\in (0,\pi)$ such that $\cos(\theta)>\delta_2 \sin(\theta).$ Then as before, for all $i\in A_s^-,$ we have, $\cos(\theta-\phi_i)<\alpha_s\cos(\theta-\phi_s).$ Moreover, $j\in A_s^0$ implies that 
	\[ \cos(\theta)\Big(\cos(\phi_j)-\alpha_s \cos(\phi_s)\Big)=0<\sin(\theta) \Big(\sin(\phi_j)-\alpha_s \sin(\phi_s)\Big),\]
	that is, $\cos(\theta-\phi_j)<\alpha_s\cos(\theta-\phi_s),$ 
	which contradicts \eqref{eq-gen}.\\
	
	If $\sin(\phi_j)- \alpha_s\sin(\phi_s)<0,$ choose $\theta\in (\pi,2\pi)$ such that $\cos(\theta)>\delta_1 \sin(\theta).$  Then for all $i\in A_{s}^-,$
	\begin{eqnarray*}
		\cos(\theta)&>&\delta_1 \sin(\theta)\\
		\Rightarrow \cos(\theta)&>& \sin(\theta) \frac{\sin(\phi_i)-\alpha_s \sin(\phi_s)}{\cos(\phi_i)-\alpha_s \cos(\phi_s)}\\
		\Rightarrow \cos(\theta-\phi_i)&<&\alpha_s\cos(\theta-\phi_s).
	\end{eqnarray*}
	On the other hand,  $j\in A_s^0$ implies that 
	\[ \cos(\theta)\Big(\cos(\phi_j)-\alpha_s\cos(\phi_s)\Big)=0<\sin(\theta) \Big(\sin(\phi_j)-\alpha_s \sin(\phi_s)\Big),\]
	that is, $\cos(\theta-\phi_j)<\alpha_s\cos(\theta-\phi_s).$ Therefore, 
	\[\max\{\cos(\theta-\phi_i):i\in C_{11}\cap C_{21}\}<\max\{\alpha_s\cos(\theta-\phi_s):s\in C\},\]
	which contradicts \eqref{eq-gen}.\\
	Thus from Case b1 and Case b2, it follows that $|A_s^0|\geq 2.$ Now assume for all $i,j\in A_s^0,$
	$$\Big(\sin(\phi_i)-\alpha_s\sin(\phi_s)\Big)\Big(\sin(\phi_j)-\alpha_s\sin(\phi_s)\Big)>0.$$ Without loss of generality, assume that $\sin(\phi_i)-\alpha_s\sin(\phi_s)>0$ for all $i\in A_s^0.$ Choose $\theta\in (0,\pi)$ such that $\cos(\theta)>\delta_2 \sin(\theta).$ Then as in Case b1, 
	\[\cos(\theta-\phi_i)<\alpha_s\cos(\theta-\phi_s), \text{ for all } i\in A_s^-. \]
	Moreover, for all $i\in A_s^0,$
	\[\sin(\theta)\Big(\sin(\phi_i)-\alpha_s\sin(\phi_s)\Big)>0 =\cos(\theta)\Big(\cos(\phi_i)-\alpha_s\sin(\phi_s)\Big),\]
	that is, $\cos(\theta-\phi_i)<\alpha_s\cos(\theta-\phi_s).$ Thus
	\[\max\{\cos(\theta-\phi_i):i\in C_{11}\cap C_{21}\}<\max\{\alpha_s\cos(\theta-\phi_s): s\in C\},\]
	which contradicts \eqref{eq-gen}. Thus $(2)$ must be true.
\end{proof}

Next we obtain an equivalent condition for \eqref{eq-gen} involving only the entries of the matrix of $T$.
\begin{prop}\label{prop-gen}
Let $\mathbb{F}=\mathbb{C}.$	Then for each $\theta \in [0,2\pi),$ the inequality \eqref{eq-gen}
		holds if and only if for each $s\in C,$ either of the following holds.
		\begin{enumerate}
			\item At least one of the sets $A_s^+,A_s^-$ is empty. Moreover, $|A_s^0|\geq 2$ and there exist distinct $i,j\in  A_s^0$ such that  $$\Big(\sin(\phi_i)-\alpha_s\sin(\phi_s)\Big)\Big(\sin(\phi_j)-\alpha_s\sin(\phi_s)\Big)<0.$$
			\item Both the sets $A_s^+,A_s^-$ are non-empty. Moreover, if 
			\begin{eqnarray*}
				\gamma_{min}&=&\min\Big\{\frac{\sin(\phi_i)-\alpha_s \sin(\phi_s)}{\cos(\phi_i)-\alpha_s \cos(\phi_s)}:i\in A_s^-\Big\},	\\
				\gamma_{max}&=&\max\Big\{\frac{\sin(\phi_i)-\alpha_s \sin(\phi_s)}{\cos(\phi_i)-\alpha_s \cos(\phi_s)}:i\in A_s^+\Big\},
			\end{eqnarray*}	
			then exactly one of the following holds.
			\begin{enumerate}
				\item $\gamma_{min}=\gamma_{max}.$
				\item $\gamma_{min}>\gamma_{max},$ and there exists $i\in A_s^0$ such that $\sin(\phi_i)>\alpha_s\sin(\phi_s).$
				\item $\gamma_{min}<\gamma_{max},$ and there exists $i\in A_s^0$ such that $\sin(\phi_i)<\alpha_s\sin(\phi_s).$
			\end{enumerate}
			In particular, if $A_s^0=\emptyset,$ then $\gamma_{min}=\gamma_{max}.$
		\end{enumerate}
	\end{prop}
	\begin{proof}
		\textbf{Necessary part:} Suppose \eqref{eq-gen} holds. Let $s\in C.$ If either $A_s^+= \emptyset, $ or $A_s^-= \emptyset,$ then $(1)$ follows from Lemma \ref{lem-gen}.  Let $A_s^+\neq \emptyset, A_s^-\neq \emptyset.$ We show that $(2)$ holds. Suppose  $\gamma_{min}>\gamma_{max}.$ Assume to the contrary that either $A_s^0=\emptyset$ or for all $i\in A_s^0,$ $\sin(\phi_i)<\alpha_s\sin(\phi_s).$ Choose $\theta \in (\pi,2\pi)$ such that 
		\[\gamma_{min}\sin(\theta)<\cos(\theta)<\gamma_{max}\sin(\theta).\]
		Then 
		\begin{equation}\label{eq-g08}
			\begin{split}
				\gamma_{min}\sin(\theta)	&<\cos(\theta)<\gamma_{max}\sin(\theta)\\
				\Rightarrow \sin(\theta) \frac{\sin(\phi_j)-\alpha_s \sin(\phi_s)}{\cos(\phi_j)-\alpha_s \cos(\phi_s)} &<	\cos(\theta)\\
				&<\sin(\theta) \frac{\sin(\phi_i)-\alpha_s \sin(\phi_s)}{\cos(\phi_i)-\alpha_s \cos(\phi_s)}, ~\forall ~i\in A_s^+, j\in A_s^-\\
				\Rightarrow 	\cos(\theta)\Big(\cos(\phi_i)-\alpha_s \cos(\phi_s)\Big)&<\sin(\theta) \Big(\sin(\phi_i)-\alpha_s \sin(\phi_s)\Big), ~\forall ~i\in A_s^+\cup A_s^-\\
				\Rightarrow \cos(\theta-\phi_i)&<\alpha_s \cos(\theta-\phi_s), ~\forall ~i\in A_s^+\cup A_s^-.
			\end{split}
		\end{equation}
		Thus if $A_s^0=\emptyset,$ then since $C_{11}\cap C_{21}=A_s^+\cup A_s^-,$ the above inequality clearly contradicts \eqref{eq-gen}. On the other hand, for all $i\in A_s^0,$ $\sin(\phi_i)<\alpha_s\sin(\phi_s)$ implies that 
		\begin{equation}\label{eq-g09}
			\begin{split}
				\cos(\theta)\Big(\cos(\phi_i)-\alpha_s\cos(\phi_s)\Big)=0&<\sin(\theta) \Big(\sin(\phi_i)-\alpha_s \sin(\phi_s)\Big), \\
				\Rightarrow \cos(\theta-\phi_i)&<\alpha_s \cos(\theta-\phi_s), ~\forall ~i\in A_s^0.
			\end{split}
		\end{equation}
		Combining \eqref{eq-g08} and \eqref{eq-g09} we get a contradiction of \eqref{eq-gen}. Therefore, if $\gamma_{min}>\gamma_{max},$ then there must exist $i\in A_s^0,$ such that  $\sin(\phi_i)>\alpha_s\sin(\phi_s),$ that is, $(b)$ holds.\\
		
		Similarly, if  $\gamma_{min}<\gamma_{max},$ then for  $\theta \in (0,\pi)$ satisfying 
		\[\gamma_{min}\sin(\theta)<\cos(\theta)<\gamma_{max}\sin(\theta),\] 
		\eqref{eq-g08} holds, which contradicts \eqref{eq-gen} if $A_s^0=\emptyset.$ Moreover,  if for all $i\in A_s^0,$ $\sin(\phi_i)>\alpha_s\sin(\phi_s),$ then \eqref{eq-g09} holds, which combined with \eqref{eq-g08} contradicts \eqref{eq-gen}. Therefore, if $\gamma_{min}<\gamma_{max},$ then there must exist $i\in A_s^0,$ such that  $\sin(\phi_i)<\alpha_s\sin(\phi_s),$ that is, $(c)$ holds.\\
		
		Clearly, if $A_s^0=\emptyset,$ then $(b), (c)$ do not hold, and so $(a)$ must be true.\\
		
		\textbf{Sufficient part:} Let $s\in C.$ We show that for all $\theta\in [0,2\pi),$
		\begin{equation}\label{eq-g10}
			\max\Big\{\cos(\theta-\phi_i):i\in C_{11}\cap C_{21}\Big\}\geq \alpha_s\cos(\theta-\phi_s).
		\end{equation}
		Suppose $(1)$ holds. Then there exist distinct $i,j\in A_s^0$ such that 
		\[\Big(\sin(\phi_i)-\alpha_s\sin(\phi_s)\Big)>0, \quad \Big(\sin(\phi_j)-\alpha_s\sin(\phi_s)\Big)<0.\]
		If $\theta\in [0,\pi),$ then 
		\[\sin(\theta)\Big(\sin(\phi_j)-\alpha_s\sin(\phi_s)\Big)\leq 0\quad \Rightarrow \quad  \cos(\theta-\phi_j)\geq \alpha_s\cos(\theta-\phi_s).\]
		If $\theta\in [\pi,2\pi),$ then 
		\[\sin(\theta)\Big(\sin(\phi_i)-\alpha_s\sin(\phi_s)\Big)\leq 0\quad \Rightarrow \quad  \cos(\theta-\phi_i)\geq \alpha_s\cos(\theta-\phi_s).\]
		Therefore, for all $\theta\in [0,2\pi),$ \eqref{eq-g10} holds.\\
		Now suppose $(2)$ holds. Let 
		\begin{eqnarray*}
			\gamma_{min}&=&\frac{\sin(\phi_l)-\alpha_s \sin(\phi_s)}{\cos(\phi_l)-\alpha_s \cos(\phi_s)},~\text{ for some } l\in A_s^-,	\\
			\gamma_{max}&=&\frac{\sin(\phi_j)-\alpha_s\sin(\phi_s)}{\cos(\phi_j)-\alpha_s \cos(\phi_s)},~\text{ for some }j\in A_s^+.
		\end{eqnarray*}
		First assume $(a)$ holds. 
		For $\theta\in [0,2\pi),$ if $\cos(\theta)<\gamma_{\max}\sin\theta=\gamma_{min}\sin(\theta),$ then 
		\begin{eqnarray*}
			\cos(\theta) &<& \sin(\theta)\frac{\sin(\phi_l)-\alpha_s \sin(\phi_s)}{\cos(\phi_l)-\alpha_s \cos(\phi_s)}\\
			\Rightarrow \cos(\theta-\phi_l)&>&\alpha_s\cos(\theta-\phi_s).
		\end{eqnarray*}
		Similarly, if $\cos(\theta)\geq \gamma_{\max}\sin\theta,$ then $ \cos(\theta-\phi_j)>\alpha_s\cos(\theta-\phi_s).$ \\
		Now assume $(b)$ holds. Let $\theta\in[0,\pi].$ 
		If $\cos(\theta)\geq \gamma_{min}\sin\theta,$ then clearly $$\cos(\theta)\geq \gamma_{max}\sin\theta \quad \Rightarrow \quad\cos(\theta-\phi_j)\geq \alpha_s\cos(\theta-\phi_s).$$ If $\cos(\theta)< \gamma_{min}\sin\theta,$ then $\cos(\theta-\phi_l)> \alpha_s\cos(\theta-\phi_s).$ If $\theta\in (\pi,2\pi),$ then from $(b),$ and $\sin(\theta)<0,$ it follows that $$\sin(\theta)\Big(\sin(\phi_i)-\alpha_s\sin(\phi_s)\Big)<0=\cos(\theta)\Big(\\cos(\phi_i)-\alpha_s\cos(\phi_)\Big),$$
		consequently, $\cos(\theta-\phi_i)>\alpha_s\sin(\theta-\phi_s).$\\
		Now assume $(c)$ holds. Let $\theta\in(\pi,2\pi).$ 
		If $\cos(\theta)\geq \gamma_{min}\sin\theta,$ then clearly 
		$$\cos(\theta)> \gamma_{max}\sin\theta \quad \Rightarrow \quad   \cos(\theta-\phi_j)> \alpha_s\cos(\theta-\phi_s).$$
		If $\cos(\theta)< \gamma_{min}\sin\theta,$ then $\cos(\theta-\phi_l)> \alpha_s\cos(\theta-\phi_s).$ If $\theta\in [0,\pi],$ then from $(c),$ and $\sin(\theta)\geq 0,$ it follows that $$\sin(\theta)\Big(\sin(\phi_i)-\alpha_s\sin(\phi_s)\Big)\leq0=\cos(\theta)\Big(\\cos(\phi_i)-\alpha_s\cos(\phi_s)\Big),$$
		consequently, $\cos(\theta-\phi_i)\geq\alpha_s\sin(\theta-\phi_s).$\\
		Therefore, if $(2)$ holds, then for all $\theta\in [0,2\pi),$ \eqref{eq-g10} holds.
	\end{proof}
	
	Using  Proposition \ref{prop-gen}, and Corollary \ref{cor-03}, it is now easy to construct linear maps $T:\ell_1^2\to \ell_\infty^m$ preserving parallel pairs for complex scalars.
	\begin{example}
		Suppose $T:\ell_1^2\to \ell_\infty^3$ is given by $T=\begin{bmatrix}
			1&e^{i\frac{\pi}{3}}\\
			1&e^{i\frac{5\pi}{3}}\\
			1&\frac{1}{2}
		\end{bmatrix}.$ Then $$k_1=k_2=1,\quad \phi_1=\frac{\pi}{3},\phi_2=\frac{5\pi}{3},\phi_3=0, \text{ and}$$
	$$C_{11}\cap C_{21}=\{1,2\}, \text{ and }C=C_{11}\cap C_{22}=\{3\}.$$ For $s=3\in C,~ \alpha_3=|t_{32}|=\frac{1}{2}.$	Therefore, 
		$\cos(\phi_1)=\cos(\phi_2)=\frac{1}{2}=\alpha_3\cos(\phi_3).$ So $A_3^0=\{1,2\}, A_3^+=A_3^-=\emptyset.$ Moreover, for $1,2\in A_3^0,$ $$\Big(\sin(\phi_1)-\alpha_3\sin(\phi_3)\Big)\Big(\sin(\phi_2)-\alpha_3\sin(\phi_3)\Big)=-\frac{3}{4}<0.$$
		Therefore, from Proposition \ref{prop-gen} it follows that  $T$ satisfies \eqref{eq-gen}, and so by Corollary \ref{cor-03}, $T$ preserves parallel pairs. 
	\end{example}

Our final goal in this section is to characterize the TEA pair preservers in $\mathcal{L}(\ell_1(J),\mathcal{Y}),$ where either $\mathcal{Y}=\ell_\infty^m$ or $\mathcal{Y}$ is strictly convex.   To consider the case $\mathcal{Y}=\ell_\infty^m,$ we need the following lemma.

\begin{lemma}\label{lem-tea}
Assume $\mathbb{F}=\mathbb{C}.$	Suppose $T\in \mathcal{L}(\ell_1^2, \ell_\infty^m)$ preserves parallel pairs, and $rank(T)=2.$ Let $C\neq \emptyset.$ 
Assume $(\alpha,\beta)\in \ell_1^2,$ where $\alpha>0,\beta=|\beta|e^{i\theta}\neq 0,$ and $$|\beta|^2\notin\Big\{\frac{k_1^2\alpha^2}{k_2^2-|t_{i2}|^2}, \frac{\alpha^2(k_1^2-|t_{j1}^2|)}{k_2^2}:i\in C_{22},j\in C_{12}\Big\}.$$ 
Suppose $T(\alpha,\beta)$ attains its norm at the $r$-th coordinate. Then $t_{r1}t_{r2}\neq 0.$
\end{lemma}
\begin{proof}
If $r\in C_{11}\cap C_{21},$ then the conclusion follows immediately. Suppose $r\in C.$ Assume to the contrary that $t_{r1}t_{r2}=0.$ Then either of the following holds:
\begin{enumerate}
	\item $r\in C_{12}\cap C_{21},$ and $t_{r1}=0.$
	\item $r\in C_{11}\cap C_{22},$ and $t_{r2}=0.$
	\item $r\in C_{12}\cap C_{22},$ and either $t_{r1}=0$ or $t_{r2}=0.$
\end{enumerate}
By Theorem  \ref{th-com}, first suppose  \eqref{eq-06} holds and $B_\theta\neq \emptyset$. From the proof of Theorem  \ref{th-com}, it follows that $T(\alpha,\beta)$ attains its norm for some $j\in A_\theta.$   Assume $(1)$ holds. Then by  \eqref{eq-06}, $\cos(\theta+\phi_j)\geq 0.$
 If $T(\alpha,\beta)$ also attains its norm at the $r$-th coordinate, then $|t_{j1}\alpha+t_{j2}\beta|=|t_{r2}\beta|$ implies that $k_1\alpha+2k_2|\beta|\cos(\theta+\phi_j)=0,$ which is a contradiction. Similarly if $(2)$ holds, then $T(\alpha,\beta)$ does not attain its norm at the $r$-th coordinate. Now assume  $(3)$ holds. First suppose $r\in  B_{\theta}$ and $t_{r1}=0.$ Note that $\cos(\theta+\phi_j)<0.$ Moreover, $\beta_r^2\leq \zeta_r\eta_r$ implies that $$k_1^2k_2^2 \cos^2(\theta+\phi_j)\leq k_1^2(k_2^2-|t_{r2}|^2)\Rightarrow k_2^2\sin^2(\theta+\phi_j)\geq |t_{r2}|^2.$$  If $T(\alpha,\beta)$ also attains its norm at the $r$-th coordinate, then  
\begin{eqnarray*}
	|t_{j1}\alpha+t_{j2}\beta|&=&|t_{r2}\beta|\\
	\Rightarrow k_1^2\alpha^2+k_2^2|\beta|^2+2k_1k_2|\alpha\beta|\cos(\theta+\phi_j)&=&|t_{r2}\beta|^2\leq |\beta|^2k_2^2\sin^2(\theta+\phi_j)\\
	\Rightarrow (k_1\alpha+k_2|\beta|\cos(\theta+\phi_j))^2&\leq &0\\
		\Rightarrow k_1\alpha+k_2|\beta|\cos(\theta+\phi_j)&=&0.
\end{eqnarray*}
Thus
\[|t_{r2}\beta|=|t_{j1}\alpha+t_{j2}\beta|=|k_1\alpha+k_2|\beta|e^{i(\theta+\phi_j)}|=k_2|\beta\sin(\theta+\phi_j)|.\]
From the last two equations, we get 
$|\beta|^2=\frac{k_1^2\alpha^2}{k_2^2-|t_{r2}|^2},$ a contradiction. On the other hand, if  $r\in B_{\theta}$ and $t_{r2}=0,$ then proceeding similarly,  we get $|\beta|^2=\frac{\alpha^2(k_1^2-|t_{r1}^2|)}{k_2^2},$ which is again a contradiction. Finally, assume $(3)$ holds, and $r\notin  B_{\theta}.$ Then $\cos(\theta+\phi_j)\geq0.$  If $T(\alpha,\beta)$ also attains its norm at the $r$-th coordinate, then $|t_{j1}\alpha+t_{j2}\beta|=|t_{r2}\beta|$ implies that $k_1\alpha+2k_2|\beta|\cos(\theta+\phi_j)<0,$ which is clearly a contradiction. \\

Now, suppose that  \eqref{eq-06} holds and  $B_\theta=\emptyset$, that is, \eqref{eq-gen} holds. Then it is easy to check  that $T(\alpha,\beta)$ attains its norm for some $s\in C_{11}\cap C_{21}.$ As previous, we can show that $T(\alpha,\beta)$ does not attain its norm at the $r$-th coordinate.
\end{proof}
Before the final result, we characterize the non-zero TEA pair preservers in  $\mathcal{L}(\ell_1^n,\mathcal{Y}).$ 
\begin{theorem}\label{th-tea}
 Suppose either of the following holds:
\begin{itemize}
	\item $\mathcal{Y}$ is strictly convex. 
	\item $\mathcal{Y}=\ell_\infty^m,$ and $\mathbb{F}=\mathbb{C}.$	
	\end{itemize}  	Then a non-zero map $T\in \mathcal{L}(\ell_1^n,\mathcal{Y})$ preserves TEA pairs if and only if $\Lambda=\{1\leq i\leq n:Te_i\neq 0\}$ is singleton.
\end{theorem}
\begin{proof}
\textbf{Necessary part:} Suppose $T:\ell_1^n\to\mathcal{Y}$ preserves TEA pairs, and $T\neq 0.$ Then $T$  preserves parallel pairs.\\
\textbf{Step 1:} In this step we prove that $rank(T)=1.$ This holds trivially, if $\mathcal{Y}$ is strictly convex. Suppose $\mathcal{Y}=\ell_\infty^m,$ and $\mathbb{F}=\mathbb{C}.$	 Assume to the contrary that $rank(T)\geq 2.$ Without loss of generality, we may assume that  $T=(t_{ij}),$ where $\{Te_1,Te_2\}$ is linearly independent. Let $S=\begin{bmatrix}
	t_{11}&t_{12}\\
	t_{21}&t_{22}\\
	\vdots &\vdots\\
	t_{m1}&t_{m2}
\end{bmatrix}.$ 	
Since by Theorem \ref{th-li}, the mapping $P:\ell_1^2\to \ell_1^n$ defined by 
$Pe_1=e_1,Pe_2=e_2$ preserves TEA pairs, $S=TP:\ell_1^2\to \mathcal{Y}$ preserves TEA pairs, consequently, $S$ preserves parallel pairs. Without loss of generality, assume $t_{j1}\geq 0,$ and $t_{j2}=|t_{j2}|e^{i\phi_j}$ for $1\leq j\leq m.$ Choose $\theta\in [0,2\pi)\setminus\{-\phi_i+k\pi:1\leq i\leq m,k\in \mathbb{Z}\},$ and $x_1>0,y_1>0, x_2=|x_2|e^{i\theta}, y_2=|y_2|e^{i\theta}$ such that the following conditions hold:
\begin{enumerate}
	\item $|x_2|^2\notin \Big\{\frac{k_1^2x_1^2}{k_2^2-|t_{r2}|^2}, \frac{x_1^2(k_1^2-|t_{s1}^2|)}{k_2^2}:r\in C_{22},s\in C_{12}\Big\}$
	\item $\frac{x_1}{|x_2|}>\frac{y_1}{|y_2|}>\Big\{\frac{|t_{r2}|}{t_{r1}}:t_{r1}t_{r2}\neq 0,1\leq r\leq m\Big\}$
	\end{enumerate} Clearly, $\Big((x_1,x_2),(y_1,y_2)\Big)$ is a TEA pair in $\ell_1^2,$ and so is $\Big(S(x_1,x_2),S(y_1,y_2)\Big)$  in $\ell_\infty^m.$ 
Thus there exists $1\leq r\leq m$ such that $S(x_1,x_2)$ and $S(y_1,y_2)$ attain their norm at the $r$-th coordinate, say $(S(x_1,x_2))_r, $ $(S(y_1,y_2))_r,$ respectively. Moreover,  $\overline{(S(x_1,x_2))_r}(S(y_1,y_2))_r\geq 0,$ that is, $\arg((S(x_1,x_2))_r)=\arg((S(y_1,y_2))_r)=\mu,$ say. By Lemma \ref{lem-tea}, $t_{r1}t_{r2}\neq 0.$ From 
$$(S(x_1,x_2))_r=t_{r1}x_1+|t_{r2}x_2| e^{i(\theta+\phi_r)}, \text{ and } (S(y_1,y_2))_r=t_{r1}y_1+|t_{r2}y_2| e^{i(\theta+\phi_r)},$$ it follows that
\begin{equation}\label{eq-tea}
\tan(\mu)=\frac{|t_{r2}x_2|\sin(\theta+\phi_r)}{t_{r1}x_1+|t_{r2}x_2|\cos(\theta+\phi_r)}=\frac{|t_{r2}y_2|\sin(\theta+\phi_r)}{t_{r1}y_1+|t_{r2}y_2|\cos(\theta+\phi_r)}.
\end{equation}
Note that, by the choice of $\theta,$ $\sin(\theta+\phi_r)\neq 0.$ Moreover,
$$t_{r1}x_1+|t_{r2}x_2|\cos(\theta+\phi_r)\neq 0,$$ for otherwise, $|\cos(\theta+\phi_r)|=\frac{t_{r1}x_1}{|t_{r2}x_2|}>1,$ by $(2),$ a contradiction. Similarly, $$t_{r1}y_1+|t_{r2}y_2|\cos(\theta+\phi_r)\neq 0.$$
Now \eqref{eq-tea} implies that $\frac{x_1}{|x_2|}=\frac{y_1}{|y_2|},$ which contradicts $(2).$ This proves the claim that $rank(T)= 1.$ \\
\textbf{Step 2:}  Without loss of generality, assume that $Te_1\neq 0.$ Let $2\leq i\leq n,$ and $Te_i=aTe_1.$ Then $x=ae_1-e_i\in \ker(T).$  We show that $a=0.$  Suppose to the contrary that $a=|a|e^{i\theta}\neq 0.$ Choose $0<\epsilon<|a|.$ Observe that $(x+\epsilon e^{i\theta}e_1,x-\epsilon e^{i\theta}e_1)$  is a TEA pair in $\ell_1^n,$ and so $(Tx+\epsilon e^{i\theta}Te_1,Tx-\epsilon e^{i\theta}Te_1)=(\epsilon e^{i\theta}Te_1,-\epsilon e^{i\theta}Te_1)$ is a TEA pair in $\mathcal{Y}.$ Therefore,
\[0=\|\epsilon e^{i\theta}Te_1-\epsilon e^{i\theta}Te_1\|=\|\epsilon e^{i\theta}Te_1\|+\|-\epsilon e^{i\theta}Te_1\|,\]
which implies that $Te_1=0,$ a contradiction. Thus  $Te_i=0$ for all $2\leq i\leq n.$\\

\textbf{Sufficient part:} Suppose $Te_j\neq 0$ for some $1\leq j\leq n,$ and $Te_i=0$ for all $1\leq i\leq n, i\neq j.$ Let $(x,y)$ be a TEA pair in $\ell_1^n,$  and $x=(x_1,\ldots,x_n), y=(y_1,\ldots,y_n).$ Then $\overline{x_j}y_j\geq 0$ for all $1\leq j\leq n.$ 
Now, $$\|Tx+Ty\|=|x_j+y_j|\|Te_j\|=(|x_j|+|y_j|)\|Te_j\|=\|Tx\|+\|Ty\|$$ implies that $(Tx,Ty)$ is a TEA pair in $\mathcal{Y}.$ Therefore, $T$ preserves TEA pairs. 
\end{proof}
 The next characterization of the non-zero TEA pair preservers in  $\mathcal{L}(\ell_1(J),\mathcal{Y})$ follows from the last theorem.
\begin{cor}\label{cor-teainf}
	 Suppose either of the following holds:
	\begin{itemize}
		\item $\mathcal{Y}$ is strictly convex. 
		\item $\mathcal{Y}=\ell_\infty^m,$ and $\mathbb{F}=\mathbb{C}.$	
	\end{itemize}  Then a non-zero map $T\in \mathcal{L}(\ell_1(J),\mathcal{Y})$ preserves TEA pairs if and only if $\Lambda=\{i\in J:Te_i\neq 0\}$ is singleton.
\end{cor}
\begin{proof}
	The sufficient part follows similarly as Theorem \ref{th-tea}. For the necessary part, assume $T\neq 0,$ and $T$ preserves TEA pairs. Then $T$ preserves parallel pairs. If $\mathcal{Y}$ is strictly convex, then clearly $rank(T)=1.$ Suppose $\mathcal{Y}=\ell_\infty^m.$  If $rank(T)> 1,$ then from Corollary \ref{cor-fr}, it follows that $|\Lambda|=rank(T).$ Then $S:\ell_1(\Lambda)\to \mathcal{Y}$ defined as  in Corollary \ref{cor-fr} preserves TEA pairs. By Theorem \ref{th-tea}, $\Lambda$ is singleton, contradicting that $|\Lambda|=rank(T).$ Thus in each case, $rank(T)=1.$ Now the result follows proceeding similarly as Step 2 of Theorem \ref{th-tea}.
\end{proof}
	
	Now we characterize the TEA pair preservers in $\mathcal{L}(\ell_1^2,\ell_\infty^m)$ if the scalars are real.  We show that in contrast to the complex scalars, there are TEA pair preservers of  rank $2$  in $\mathcal{L}(\ell_1^2,\ell_\infty^m)$ if the scalars are real.
	
 \begin{theorem}\label{th-tear}
 	Assume $\mathbb{F}=\mathbb{R}.$ Suppose $T\in\mathcal{L}(\ell_1^2,\ell_\infty^m)$ is non-zero.   Then $T$ preserves TEA pairs if and only if either of the following holds.
 	\begin{enumerate}
 		\item $rank(T)=1,$ and $\Lambda=\{i:1\leq i\leq 2,Te_i\neq 0\}$ is singleton.
 		\item $rank(T)=2,$ and 	$T$ is of the following form upto permutation of rows.
 		\[	T=\begin{bmatrix}
 			t_{11}&t_{12}\\
 			t_{11}&-t_{12}\\
 			t_{31}&t_{32}\\
 			\vdots&\vdots\\
 			t_{m1}&t_{m2}
 		\end{bmatrix}, ~\text{ where }|t_{11}|\geq |t_{i1}|, \text{ and }|t_{12}|\geq |t_{i2}| \text{ for all }3\leq i\leq m. \] 
 		\end{enumerate}
 \end{theorem}
\begin{proof}
	For the sufficient part, if $(1)$ holds, then proceeding similarly as Theorem \ref{th-tea}, we get $T$ preserves TEA pairs. Now suppose $(2)$ holds. Without loss of generality, assume $t_{11}>0,t_{12}>0.$ Suppose $((a,b),(c,d))$ is a TEA pair in $\ell_1^2.$ Then $ac\geq 0,bd\geq 0.$  If $a,c,b,d\geq 0,$ then both $T(a,b),$ and $T(c,d)$ attain their norm at the first coordinate. Since the first coordinates of  $T(a,b)$ and $T(c,d)$ are non-negative, $\Big(T(a,b),T(c,d)\Big)$ is a TEA pair.   If $a,c\geq 0, b,d\leq 0,$ then both $T(a,b),$ and $T(c,d)$ attain their norm at the second coordinate. Since the second coordinates of  $T(a,b)$ and $T(c,d)$ are non-negative, $\Big(T(a,b),T(c,d)\Big)$ is a TEA pair. Similarly, for the other cases, it is straightforward to check that  $\Big(T(a,b),T(c,d)\Big)$ is a TEA pair.   \\
	
	Conversely, assume $T=(t_{ij})$ preserves TEA pairs. If $rank(T)=1,$ then proceeding similarly as step (2) of Theorem \ref{th-tea}, we get $\Lambda$ is singleton. Suppose $rank(T)=2.$ Without loss of generality, assume $t_{i1}\geq 0$ for all $1\leq i\leq m.$ We claim that $t_{i2}>0,t_{j2}<0$ for some $i,j\in C_{11}\cap C_{21}.$ For otherwise, by Theorem \ref{th-comr}, $sgn(t_{r2})$ is constant for all $1\leq r\leq m$ such that $t_{r1}t_{r2}\neq 0.$ Let $t_{r2}>0$ for all $r$ satisfying $t_{r1}t_{r2}\neq 0.$ Choose $b,d<0$ such that $|b|$ is sufficiently large and $|d|$ is sufficiently small. Then $\Big((1,b),(1,d)\Big)$ is a TEA pair in $\ell_1^2.$ Since all coordinates of $T(1,b)$ are negative, and all  coordinates of $T(1,d)$ are positive,   $\Big(T(1,b),T(1,d)\Big)$ is not a TEA pair in $\ell_\infty^m,$ a contradiction. So assume $1,2\in C_{11}\cap C_{21}$ and $t_{12}>0,t_{22}<0.$ Therefore $t_{11}=t_{21}\geq |t_{i1}|,$ and $ -t_{22}=t_{12}\geq |t_{i2}|$ for all $3\leq i\leq m.$
	\end{proof}

Finally we conclude the section with the characterization of TEA pair preservers in $\mathcal{L}(\ell_1^2,\mathcal{Y}),$ where the scalars are real, and $\mathcal{Y}$ is strictly convex. The proof follows similarly as Theorem \ref{th-n2p}. To avoid monotony, we skip the proof.  

\begin{theorem}\label{th-ltea}
Assume $\mathbb{F}=\mathbb{R}.$ Suppose $T\in\mathcal{L}(\ell_1^2,\mathcal{Y})$ is non-zero.   Then $T$ preserves TEA pairs if and only if $T(\cdot)=f(\cdot)y,$ where $y\in \mathcal{Y}$ and $f\in \{e_1,e_2\}\subset \ell_\infty^2.$

\end{theorem}

\section{Anknowledgement}
The author would like to thank DST, Govt. of India for partial financial support in the form of INSPIRE Faculty Fellowship (DST/INSPIRE/04/2022/001207).

\bibliographystyle{amsplain}

\end{document}